\documentclass[a4paper,10pt,american,reqno]{amsart}

\usepackage{dsfont}
\usepackage{stmaryrd}
\usepackage{esint}
\usepackage{amssymb}
\usepackage[square,sort,comma,numbers]{natbib}
\usepackage{enumitem}

\setlist[1]{left=0pt .. \parindent}
\setlist{itemindent=.5em}

\makeatletter
\newsavebox{\@brx}
\newcommand{\llangle}[1][]{\savebox{\@brx}{\(\m@th{#1\langle}\)}%
  \mathopen{\copy\@brx\kern-0.5\wd\@brx\usebox{\@brx}}}
\newcommand{\rrangle}[1][]{\savebox{\@brx}{\(\m@th{#1\rangle}\)}%
  \mathclose{\copy\@brx\kern-0.5\wd\@brx\usebox{\@brx}}}
\makeatother

\usepackage[latin1]{inputenc}
\usepackage{xcolor}
\usepackage{verbatim}

\usepackage[colorlinks,pdfpagelabels,pdfstartview = FitH,bookmarksopen
= true,bookmarksnumbered = true,linkcolor = blue,plainpages =
false,hypertexnames = false,citecolor = red,pagebackref=false]{hyperref}

\allowdisplaybreaks

\newtheorem{theorem}{Theorem}
\newtheorem{lemma}[theorem]{Lemma}
\newtheorem{proposition}[theorem]{Proposition}

\theoremstyle{definition}
\newtheorem{definition}[theorem]{Definition}

\theoremstyle{remark}
\newtheorem{remark}[theorem]{Remark}

\numberwithin{theorem}{section}
\numberwithin{equation}{section}

\newcommand{\mollifytime}[2]{[\![ #1 ]\!]_{#2}}
\newcommand{\N}{\mathbb{N}}
\newcommand{\R}{\mathbb{R}}

\renewcommand{\d}{\mathrm{d}}
                  
\newcommand{\dx}{\mathrm{d}x}

\newcommand{\dt}{\mathrm{d}t}
\newcommand{\ds}{\mathrm{d}s}

\newcommand{\eps}{\varepsilon}
\renewcommand{\epsilon}{\varepsilon}
\renewcommand{\rho}{\varrho}
\renewcommand{\u}{\boldsymbol{u}}
\newcommand{\g}{\boldsymbol{g}}

\newcommand{\muplus}{\boldsymbol{\mu}^+}
\newcommand{\muminus}{\boldsymbol{\mu}^-}
\newcommand{\mupm}{\boldsymbol{\mu}^\pm}
\newcommand{\bomega}{\boldsymbol{\omega}}

\DeclareMathAlphabet{\mathpzc}{OT1}{pzc}{m}{it}

\DeclareMathOperator{\spt}{spt}

\DeclareMathOperator{\dist}{dist}

\DeclareMathOperator*{\osc}{osc}
\DeclareMathOperator{\Div}{div}
\DeclareMathOperator*{\esssup}{sup}

\DeclareMathOperator{\loc}{loc}

\makeatletter
\@namedef{subjclassname@2020}{%
  \textup{2020} Mathematics Subject Classification}
\makeatother

\def\Xint#1{\mathchoice
    {\XXint\displaystyle\textstyle{#1}}%
    {\XXint\textstyle\scriptstyle{#1}}%
    {\XXint\scriptstyle\scriptscriptstyle{#1}}%
    {\XXint\scriptscriptstyle\scriptscriptstyle{#1}}%
    \!\int}
\def\XXint#1#2#3{\setbox0=\hbox{$#1{#2#3}{\int}$}
    \vcenter{\hbox{$#2#3$}}\kern-0.5\wd0}
\def\bint{\Xint-}
\def\dashint{\Xint{\raise4pt\hbox to7pt{\hrulefill}}}

\def\XXiint#1#2#3{\setbox0=\hbox{$#1{#2#3}{\iint}$}
    \vcenter{\hbox{$#2#3$}}\kern-0.5\wd0}


\begin{document}
%
%

\title[Boundary regularity for obstacle problems]{Continuity up to the boundary for obstacle problems to porous medium type equations}

\author[K. Moring]{Kristian Moring}
\address{Kristian Moring\\
	Fakult\"at f\"ur Mathematik, Universit\"at Duisburg-Essen\\
	Thea-Ley\-mann-Str.~9, 45127 Essen, Germany}
\email{kristian.moring@uni-due.de}

\author[L. Sch\"atzler]{Leah Sch\"atzler}
\address{Leah Sch\"atzler \\ 
Fachbereich Mathematik, Paris-Lodron-Universit\"at Salzburg \\
Hellbrunner Str~34, 5020 Salzburg, Austria}
\email{leahanna.schaetzler@plus.ac.at}

\begin{abstract}
We show that signed weak solutions to obstacle problems for porous medium type equations with Cauchy-Dirichlet boundary data are continuous up to the parabolic boundary, provided that the obstacle and boundary data are continuous.
This result seems to be new even for signed solutions to the (obstacle free) Cauchy-Dirichlet problem to the singular porous medium equation, which is retrieved as a special case.
\end{abstract}
\maketitle

\section{Introduction}
In the present paper, we consider the obstacle problem to partial differential equations whose prototype is the porous medium equation (PME)
$$
	\partial_t \big( |u|^{q-1} u \big) - \Delta u = 0,
$$
where $q \in (0,\infty)$. We refer to $0<q<1$ as the degenerate case, and to $ q > 1$ as the singular case. In the literature, the singular porous medium equation is often called the fast diffusion equation. For the standard theory of equations of this type, see e.g.~\cite{Vazquez_PME,Vazquez_FDE,DGV,DK,WZY}. More specifically, we consider problems of the form
\begin{align} \label{eq:pde}
\left\{
\begin{array}{rl}
\partial_t \big( |u|^{q-1} u \big)- \Div \mathbf{A}(x,t,|u|^{q-1} u, \nabla u) = 0 &\text{ in } \Omega_T, \vspace{1mm}\\
u \geq \psi &\text{ in } \Omega_T, \vspace{1mm}\\
u=g &\text{ on } S_T, \vspace{1mm}\\
u=g_o &\text{ on } \Omega \times \{0\},
\end{array}
\right. 
\end{align}
where $\mathbf{A}$ is a vector field satisfying the assumptions specified in Section~\ref{subsec:def-sol}, and $S_T = \partial \Omega \times (0,T)$ denotes the lateral boundary of $\Omega_T$.

In nonlinear potential theory, the obstacle problem is used as a standard tool. Supersolutions, which are defined in potential theory via a parabolic comparison principle, are connected to weak supersolutions in certain ways. In particular, they are typically approximated by weak supersolutions. Such approximants can be constructed via successive obstacle problems, in which regularity up to the boundary plays an important role, see e.g.~\cite{HKM} for the elliptic case and~\cite{KL-crelle} for the porous medium equation.

Local H\"older continuity for nonnegative solutions to the (obstacle free) porous medium equation was proven by DiBenedetto and Friedman in~\cite{DF} in the degenerate case. For the corresponding result in the singular case see~\cite{DGV}. Signed solutions to the porous medium type equations were treated in~\cite{Liao}.

In presence of an obstacle, local H\"older continuity was shown for quasilinear equations in~\cite{StruweVivaldi}, and problems of quadratic growth were treated in~\cite{Choe}. For the porous medium equation, an analogous result was shown in~\cite{BLS_holder} in the degenerate regime in case of nonnegative obstacles. The corresponding result in the singular case was proven in~\cite{Cho_Scheven}, where also more general equations of porous medium type were included. The local H\"older continuity in case of signed obstacles has been treated in~\cite{MS}, in which the result covered the full range of the parameter $q$.

In the obstacle free case continuity up to the boundary for solutions to degenerate porous medium type problems of the form~\eqref{eq:pde}$_{(1,3,4)}$ was proven in~\cite{DB-boundary}. Boundary regularity for nonnegative solutions to the degenerate model equation in terms of Perron's method  was developed in~\cite{KLL}. The results were extended in~\cite{BB}, where the authors proved a barrier characterization of regular boundary points, including also domains that are non-cylindrical. Furthermore, it is worth to mention that existence of a (unique) solution that is continuous up to the boundary has been shown in~\cite{Abdulla1,Abdulla2}, in which the author considers nonnegative very weak solutions to the model equation in non-cylindrical domains for the full parameter range. Boundary regularity for doubly nonlinear equations was shown in~\cite{part2} in the doubly degenerate case. Several tools in the aforementioned paper are also adapted and applied in this paper. 

Questions on existence related to obstacle problems for porous medium type equations are addressed in~\cite{AL,BLS,Schaetzler1,Schaetzler2}.

At this stage, we state our main result.

\begin{theorem}
\label{thm:main_theorem}
	Let $q \in (0,\infty)$ and $\Omega \subset \R^n$ be a bounded open set which satisfies the geometric density condition~\eqref{geometry} for some $\alpha_*, \rho_o >0$. Furthermore, suppose that the obstacle function $\psi$ satisfies~\eqref{eq:psi_conds}, the lateral boundary datum $g$ satisfies~\eqref{eq:g-conds} and the initial datum $g_o$~\eqref{eq:go-conds}. Let $u$ be a weak solution to the obstacle problem~\eqref{eq:pde} according to Definition~\ref{d.obstacle-wsol}. Then $u$ can be extended as a function that is  continuous up to the boundaries $S_T$ and $\Omega \times \{0\}$. More precisely, if $\mathcal K \subset \overline{\Omega} \times (0,T)$ is a compact set, there exists a modulus of continuity $\bomega$ determined by $C_o$, $C_1$, $n$, $q$, $\alpha_*$, $\rho_o$, $\|u\|_\infty$, $\bomega_\psi$, $\bomega_g$ and $\dist(\mathcal K, \Omega \times \{0\})$ such that 
\begin{align*}
|u(x_o,t_o) - u(x_1,t_1) | \leq \bomega \left( |x_o-x_1| + \|u\|_\infty^\frac{1-q}{2}|t_o-t_1|^\frac{1}{2} \right)
\end{align*}
for every $(x_o,t_o) \in \mathcal K \cap S_T$ and $(x_1,t_1) \in \mathcal K$. 

Furthermore, if $K \subset \Omega$ is a compact set, there exists a modulus of continuity $\bomega$ determined by $C_o$, $C_1$, $n$, $q$, $\|u\|_\infty$, $\bomega_\psi$, $\bomega_{g_o}$ and $\dist (K,\partial \Omega)$ such that 
\begin{align*}
|u(x_o,0) - u(x_1,t_1) | \leq \bomega \left( |x_o-x_1| + \|u\|_\infty^\frac{1-q}{2}|t_1|^\frac{1}{2} \right)
\end{align*}
for every $x_o \in K$ and $(x_1,t_1) \in K \times [0,T)$.

If $\psi$ satisfies~\eqref{eq:psi_conds}, $g \in C(\Omega_T \cup \partial_p \Omega_T)$, where $\partial_p \Omega_T = (\overline{\Omega} \times \{0\}) \cup S_T$, satisfies~\eqref{eq:g-conds} and the initial datum $g_o$ satisfies~\eqref{eq:go-conds} with $g_o \equiv g(\cdot,0)$, then $u$ can be extended as a function in $\Omega_T \cup \partial_p \Omega_T$ that is continuous up to $\partial_p \Omega_T$. Moreover, there exists a modulus of continuity $\bomega$ determined by $C_o$, $C_1$, $n$, $q$, $\alpha_*$, $\rho_o$, $\|u\|_\infty$, $\bomega_\psi$, $\bomega_g$ such that 
\begin{align*}
|u(x_o,t_o) - u(x_1,t_1) | \leq \bomega \left( |x_o-x_1| + \|u\|_\infty^\frac{1-q}{2}|t_o-t_1|^\frac{1}{2} \right)
\end{align*}
for every $(x_o,t_o) \in \partial_p \Omega_T$ and $(x_1,t_1) \in \Omega_T \cup \partial_p \Omega_T$.
\end{theorem}

\subsection{Strategy}

The strategy of the proof can roughly be divided into two parts. In Section~\ref{sec:def-prop}, we show that with appropriate truncation levels the truncated solutions to the obstacle problem are weak sub- or supersolutions to the obstacle free problem. Together with a suitable extension argument we are able to treat the truncated solutions as local sub- or supersolutions to the obstacle free problem also in cylinders that intersect the complement of the domain. In this case, admissibility of the truncation levels depends on the extrema of both the obstacle and the boundary datum.
 
In the second part we show continuity up to the parabolic boundary, which we prove separately up to the lateral and initial boundaries. Near the lateral boundary we consider backward cylinders whose vertices are attached to the lateral boundary, while near the initial boundary we use forward cylinders attached to the initial boundary. The proofs are based on estimates for sub- and supersolutions to the obstacle free problem in both cases. 

We separate the cases where the solution is near zero and where it is away from zero. In the former case, extrema of the solution are controlled by the oscillation of the solution, obstacle and the boundary data in a given cylinder, while the latter is complementary to the former case.
In both of these cases, we need to split the proof into further alternatives.
Very roughly speaking, we distinguish whether quantities related to the extrema or oscillation of the solution are large or small compared to quantities related to the extrema or oscillation of the boundary values and obstacle function.
By choosing said quantities appropriately, we are able to truncate the solution, extend it to the complement, and treat the extended truncations as sub- or supersolutions to the obstacle free problem by results in Section~\ref{subsec:truncations}. Then, the reduction in oscillation is shown by using tools from the obstacle free case, which are described in Section~\ref{sec:aux}. At this point, boundary conditions together with~\eqref{geometry} obviously play an important role.

The tools in Section~\ref{sec:aux} for reduction in oscillation are divided into the local case and the case near the initial boundary, which are further divided into tools used in the cases near zero and away from zero. In the cases away from zero, we apply a De Giorgi type iteration scheme, which is essentially based on energy estimates given in the beginning of Sections~\ref{subsec:local-case} and~\ref{subsec:initial-case}. In the case near zero on the initial boundary, we use a result on propagation of positivity, see Lemma~\ref{l.de-giorgi-initial}. In the case near zero on the lateral boundary we use separate tools in the singular and degenerate cases. In the degenerate case we are able to exploit a standard De Giorgi type iteration argument, while in the singular case the same approach is not applicable as such. In that case, we are required to rely on a heavier tool; a certain formulation of expansion of positivity, see~\cite{Liao}, which is adapted to our setting in Proposition~\ref{prop:expansion-of-positivity-q}.

\subsection{On the notion of solution}

In this paper we consider a notion of solution to the obstacle problem, in which the variational inequality holds in a local sense. The boundary conditions are imposed separately and they are attained in the standard Sobolev sense slice-wise on the lateral boundary and in the $L^{q+1}$-sense on the initial boundary. The reason for using the local notion mainly is that it allows comparison maps whose boundary values do not necessarily coincide with the boundary values of the solution, which makes the proofs in Section~\ref{subsec:truncations} attainable without unnecessary complications. Furthermore, this allows us to relax the assumptions on the boundary data. It was shown in~\cite{BLS} that the local inequality implies the global one (which was used e.g. in~\cite{BLS,Schaetzler1,Schaetzler2}), provided that the domain and boundary data satisfy appropriate regularity assumptions and the boundary data is attained in the aforementioned sense. The reverse direction also holds under similar assumptions, which is discussed in Appendix~\ref{sec:appendix_definitions}.

\subsection{Novelty and significance}

To the authors' knowledge, the issue of boundary regularity has not been addressed in the literature even in the obstacle free case for signed solutions to the singular porous medium equation. Our result already covers this as a special case. Furthermore, in contrast to the results in the local case (\cite{MS,Cho_Scheven,BLS_holder}), we are able to treat an obstacle and boundary datum with general moduli of continuity. 

Indeed, the approach we use has certain differences compared to the proofs for local H\"older regularity to the obstacle problem in~\cite{MS,Cho_Scheven,BLS_holder}. In the aforementioned cases, the obstacle was involved in all parts of the proof. Furthermore, the construction of a shrinking sequence of cylinders was strongly affected by continuity properties of the obstacle. Namely, the first step of the construction and oscillation decay estimate (which was eventually proved in non-intrinsic cylinders) relied heavily on the assumption that the obstacle is H\"older continuous. In our proof of boundary regularity the assumptions on the obstacle can be relaxed. Due to the conditions on the boundary of the domain and chosen alternatives, the treatment of reduction in oscillation can be reduced entirely to a problem in the obstacle free case provided that the boundary satisfies rather mild regularity assumptions. Furthermore, this allows us to construct cylinders and prove an oscillation decay estimate in a manner which is not affected by any specific continuity properties of the obstacle or the boundary datum.

Moreover, instead of assuming that the solution is bounded in the first place, the proof is included in Section~\ref{subsec:boundedness} based on a maximum principle for subsolutions to the obstacle free problem.

\medskip
 
\noindent
{\bf Acknowledgments.}
K. Moring has been supported by the Magnus Ehrnrooth Foundation. Both authors have also been supported by the FWF-Project P31956-N32 "Doubly nonlinear evolution equations". The authors would like to thank Christoph Scheven for useful discussions and hospitality during the second author's visit at the University of Duisburg-Essen.



\section{Definition and some properties of solutions} \label{sec:def-prop}

\subsection{Notation}
Let $\Omega \subset \R^n$, $n \geq 2$, be a bounded domain and $0<T<\infty$.
For cylinders with vertex $z_o = (x_o,t_o) \in \overline{\Omega_T}$ we use the notations
\begin{align*}
	Q_{\varrho,s}(z_o) := B_\varrho(x_o) \times (t_o-s,t_o),\\
	Q_{\varrho,s}^+(z_o) := B_\varrho(x_o) \times (t_o,t_o+s).
\end{align*}

For $b \in \R$ and $\alpha > 0$ we denote the signed $\alpha$-power of $b$ by
\[ \boldsymbol{b}^\alpha :=
\begin{cases}
|b|^{\alpha -1} b &\text{if } b \neq 0,\\
0 &\text{if } b = 0. 
\end{cases}\]

\subsection{Definition of solutions} \label{subsec:def-sol}
In~\eqref{eq:pde} we assume that $\mathbf{A} \colon \Omega_T \times \R \times \R^n \to \R^n$ is a Carath\'eodory function, i.e.,~it is
measurable with respect to $(x,t) \in \Omega_T$ for all $(u, \xi) \in \R \times \R^n$
and continuous with respect to $(u, \xi) \in \R \times \R^n$ for a.e.~$(x,t) \in \Omega_T$. Moreover, we suppose that $\mathbf{A}$ satisfies the structure conditions
\begin{equation}
	\left\{
	\begin{array}{l}
		\mathbf{A}(x,t,u,\xi) \cdot \xi \geq C_o |\xi|^2, \\[5pt]
		|\mathbf{A}(x,t,u,\xi)| \leq C_1 |\xi|,
	\end{array}
	\right.
	\quad \text{for a.e. $(x,t) \in \Omega_T$ and all $(u, \xi) \in \R \times \R^n$,}
	\label{eq:structure}
\end{equation}
where $C_o, C_1 >0$ are given constants.

For a given obstacle function $\psi \colon \Omega \times [0,T) \to \R$ we define classes of functions
$$
K_{\psi}(\Omega_T) := \left\{ v\in C((0,T); L^{q+1}(\Omega)) : v \in  L^2(0,T;H^1(\Omega)),\, v \geq \psi \, \text{ a.e. in } \Omega_T \right\},
$$
and
$$
K'_{\psi}(\Omega_T) := \left\{ v \in K_{\psi}(\Omega_T): \partial_t v \in L^{q+1}(\Omega_T) \right\}.
$$
For boundary values $g \in L^2(0,T;H^1(\Omega))$ we consider the class
$$
K_{\psi,g}(\Omega_T) := \left\{ v \in K_{\psi}(\Omega_T) : v-g \in L^2(0,T; H^1_0(\Omega)) \right\}.
$$
Furthermore, we denote
\begin{equation*}
I_{\psi_o}(\Omega) := \left\{v \in L^{q+1}(\Omega) : v \geq \psi(\cdot,0)\, \text{ a.e. in } \Omega \right\}.
\end{equation*}

For the sake of generality, we define weak solutions to the obstacle problem as local solutions that attain the lateral boundary values $g$ or the initial boundary values $g_o$ in the following way.
For a discussion of the connection to global weak solutions we refer to Appendix~\ref{sec:appendix_definitions}.
\begin{definition} \label{d.obstacle-wsol}
\begin{enumerate}

\item A function $u \in K_{\psi}(\Omega_T)$ is a weak solution to the obstacle problem~\eqref{eq:pde}$_{1-2}$, if and only if the variational inequality
\begin{equation}\label{eq:obstacle-wsol}
\llangle \partial_t \u^q , \alpha \eta(v-u) \rrangle + \iint_{\Omega_T} \alpha \mathbf{A}(x,t,\u^q, \nabla u) \cdot \nabla \left(\eta (v-u) \right) \, \d x \d t \geq 0
\end{equation}
holds for every cut-off function $\alpha\in W^{1,\infty}_0([0,T]; \R_{\geq 0})$, every $\eta \in C^1_0(\Omega, \R_{\geq0})$ and all comparison maps $v \in K_{\psi}'(\Omega_T)\cap C([0,T];L^{q+1}(\Omega)) $.
Here, we define the time term by
\begin{align*} 
\llangle \partial_t \u^q , \alpha \eta(v-u) \rrangle := &\iint_{\Omega_T} \eta \left[ \alpha' \left( \tfrac{q}{q+1}|u|^{q+1} -  \u^q v \right) - \alpha \u^q \partial_t v \right] \, \d x \d t.
\end{align*}

\item We say that $u \in K_{\psi,g}(\Omega_T)$ is a weak solution to the obstacle problem with lateral boundary values $g \in K_\psi(\Omega_T)$ if~\eqref{eq:obstacle-wsol} holds.

\item Moreover, we say that $u$ is a weak solution to the obstacle problem with initial values $g_o \in I_{\psi_o}(\Omega)$ if~\eqref{eq:obstacle-wsol} holds and 
\begin{equation} \label{eq:obstacle-initial-q+1}
\bint_0^h \int_\Omega |u-g_o|^{q+1} \, \d x \d t \xrightarrow{h\downarrow 0} 0.
\end{equation} 
\end{enumerate}

\end{definition}

Throughout the paper, we will use the following notion when referring to a solution to the partial differential equation \eqref{eq:pde}$_1$ without obstacle.
\begin{definition}
\label{def:localsol_pme}
\begin{enumerate}
\item A function
$$
	u \in C^0((0,T);L^{q+1}_\mathrm{loc}(\Omega)) \cap L^2_\mathrm{loc}(0,T;W^{1,2}_\mathrm{loc}(\Omega))
$$
is a local weak sub(super)solution to \eqref{eq:pde}$_1$ in $\Omega_T$ if
\begin{align}
	\iint_{\Omega_T} 
	- \partial_t \varphi \u^q + \mathbf{A}(x,t,\u^q, \nabla u)\cdot \nabla \varphi \,\dx\dt
	\leq (\geq) \, 0
	\label{eq:localsol_pme}
\end{align}
holds for all test functions $\varphi \in C^\infty_0(\Omega_T, \R_{\geq0})$.
\item Moreover, we say that $u$ is a weak sub(super)solution to \eqref{eq:pde}$_1$ with initial values $g_o \in L^{q+1}_\mathrm{loc}(\Omega)$ if it is a local weak sub(super)solution and
\begin{equation} \label{eq:pme_initial_values}
\bint_0^h \int_K (u-g_o)_\pm^{q+1} \, \d x \d t \xrightarrow{h\downarrow 0} 0,
\end{equation}
holds true for every compact set $K \subset \Omega$.
\end{enumerate}
\end{definition}

In the proof of boundary regularity, we will make following assumptions on the obstacle $\psi$ and boundary data $g$ and $g_o$:
\begin{align}
&\psi \in C(\overline{\Omega_T}) \text{ with modulus of continuity } \bomega_{\psi} (\cdot) \text{ in } \overline{\Omega_T}; \label{eq:psi_conds} \\ 
&g \in K_\psi(\Omega_T) \cap C(\Omega_T \cup \overline{S_T}), \text{ with modulus of continuity } \bomega_{g} (\cdot) \text{ on } \overline{S_T}; \label{eq:g-conds}  \\
&g_o \in I_{\psi_o}(\Omega) \cap C(\overline{\Omega}) \text{ with modulus of continuity } \bomega_{g_o} (\cdot) \text{ on } \overline{\Omega}. \label{eq:go-conds}
\end{align}

Furthermore, we will assume that the domain $\Omega$ satisfies the geometric density condition, i.e.
\begin{equation}\label{geometry}
	\begin{minipage}[c][1.5cm]{0.9\textwidth}
	there exist $\alpha_*\in(0,1)$ and $\rho_o>0$, such that for all $x_o\in\partial \Omega$ and $\rho \in (0,\rho_o]$ there holds
	$
	|\Omega \cap B_{\rho}(x_o)|\leq (1-\alpha_*)|B_\rho(x_o)|.
	$
	\end{minipage}
\end{equation}

\subsection{Truncation of solutions to the obstacle problem} \label{subsec:truncations}
In this section, we show that truncations of weak solutions to the obstacle problem to \eqref{eq:pde} with suitable truncation levels are weak sub(super)solutions to \eqref{eq:pde} in the obstacle-free setting in the intersection of the space-time cylinder $\Omega_T$ with cylinders with vertex in $\Omega_T$ or on its lateral boundary $S_T$.

To this end, for $v \in L^1(\Omega_T)$, $v_o \in L^1(\Omega)$ and $h >0$ we define a mollification in time by
\begin{equation} \label{e.time_moll_g}
\mollifytime{v}{h} (x,t) :=  
e^{-\frac{t}{h}} v_o(x) + \tfrac{1}{h} \int_0^t e^\frac{s-t}{h}v(x,s) \, \d s.
\end{equation}

The properties of the this mollification procedure are collected in the following lemma. We refer to~\cite[Lemma 2.9]{Kinnunen-Lindqvist} and~\cite[Appendix B]{BDM:pq}. 
\begin{lemma}
\label{lem:time_mollification}
Let $1 \leq r \leq \infty$, $X \in \big\{ L^r(\Omega), W^{1,r}(\Omega), W^{1,r}_0(\Omega) \big\}$, $v \in L^r(0,T;X)$ and $\mollifytime{v}{h}$ be defined by \eqref{e.time_moll_g} with $v_o \in X$. Then the following statements hold true:
\begin{enumerate}
\item We have that $\mollifytime{v}{h} \in L^r(0,T;X)$. Moreover,  $\mollifytime{v}{h} \in C([0,T];X)$.
\item
The weak time derivative $\partial_t \mollifytime{v}{h} \in L^r(0,T;X)$ is given by the formula
\begin{equation*}
  \partial_t \mollifytime{v}{h} = \tfrac{1}{h} ( v - \mollifytime{v}{h} ).
\end{equation*}
\item If $1 \leq r < \infty$, then
\begin{equation*}
\mollifytime{v}{h} \to v\quad \text{ in } L^r(0,T;X) \text{ as } h \downarrow 0.  
\end{equation*}
Further, if $v \in C([0,T];X)$ and $v_o = v(\cdot,0)$, then $\mollifytime{v}{h} \to v$ in $C([0,T];X)$ as $h \downarrow 0$.
\item
If $\Omega \subset \R^n$ is a bounded set, $v \in C(\overline{\Omega_T})$ and $v_o = v(\cdot,0)$, then 
$$
\mollifytime{v}{h} \to v \quad \text{ uniformly in }\Omega_T\ \text{ as }\ h \to 0.
$$
\end{enumerate}
\end{lemma}

With this mollification in time at hand, we are able to proceed as follows.
\begin{lemma} \label{lem:weak_subsol}
Let $z_o = (x_o,t_o) \in \overline \Omega \times (0,T)$ and consider the cylinder $Q_{R,S}(z_o)$ with $R,S>0$.
Further, let $u$ be a weak solution to the obstacle problem according to Definition~\ref{d.obstacle-wsol} and suppose that $\psi \in C(\overline{\Omega_T})$. Then,
\begin{enumerate}
\item \label{lem:weak_subsol_1} $\min \{u,k\}$ is a weak supersolution to~\eqref{eq:pde}$_1$ in $Q_{R,S}(z_o) \cap \Omega_T$ in the sense of Definition \ref{def:localsol_pme} for any $k \in \R$;
\item \label{lem:weak_subsol_2} $\max\{u,k\}$ is a weak subsolution to~\eqref{eq:pde}$_1$ in $Q_{R,S}(z_o) \cap \Omega_T$ in the sense of Definition \ref{def:localsol_pme} for any $k \geq \sup_{Q_{R,S}(z_o) \cap \Omega_T} \psi$.
\end{enumerate}
\end{lemma}

\begin{proof}
We start by showing \eqref{lem:weak_subsol_2}. Let us denote $Q = Q_{R,S}(z_o)$ for simplicity. Since $\Omega$ is bounded, there exists $\rho > 0$ such that $B_\frac{\rho}{2}(0) \supset \Omega$.  Since $\psi \in C(\overline{\Omega_T})$, we extend it as a continuous function to $B_{\rho} (0) \times [0,T]$ and let $\psi \equiv 0$ in $(\R^n \setminus B_{\rho}(0)) \times [0,T]$. We still denote the extension by $\psi$. We use time mollifications $\mollifytime{u}{h,(\psi_o)_\delta}$ and $ \mollifytime{\psi}{h,(\psi_o)_\delta}$ with initial values $(\psi_o)_\delta$, where $\psi_o = \psi(\cdot,0)$ and $(\cdot)_\delta$ denotes a standard mollification in the spatial variables with parameter $\delta > 0$. Observe that for any sequence $\eps_i \downarrow 0$ when $i \to \infty$, there exists a sequence $\delta_i \downarrow 0$ such that $\|\mollifytime{\psi}{h,(\psi_o)_\delta} - \mollifytime{\psi}{h,\psi_o}\|_\infty < \eps_i$ for every $\delta \in (0,\delta_i]$ and $h > 0$, and there exists a sequence $\tilde h_i > 0$ such that $\|\psi - \mollifytime{\psi}{h,\psi_o}\|_\infty < \eps_i$ for every $h \in (0,\tilde h_i]$ by Lemma~\ref{lem:time_mollification} (4). 
Furthermore, as in the proof of Lemma~\ref{lem:global-is-local} together with Lemma~\ref{lem:time_mollification} (3), for the sequence $\delta_i \downarrow 0$ determined above, there exists $\hat h_i \downarrow 0$ such that $\max \{\| \mollifytime{u}{h,(\psi_o)_{\delta_i}} - u\|_{L^2(0,T;H^1(\Omega))}, \| \mollifytime{u}{h,(\psi_o)_{\delta_i}} - u\|_{L^{q+1}(\Omega_T)}\} < \eps_i$ for every $h \in (0,\hat h_i]$. By choosing $ h_i = \min \{\tilde h_i,\hat h_i\}$, we find that 
$$
\|\mollifytime{\psi}{h_i,(\psi_o)_{\delta_i}} - \psi\|_\infty \xrightarrow{i\to \infty} 0, \quad \|\mollifytime{u}{h_i,(\psi_o)_{\delta_i}} - u\|_{L^2(0,T;H^1(\Omega))} \xrightarrow{i\to \infty} 0,
$$
and
$$
\|\mollifytime{u}{h_i,(\psi_o)_{\delta_i}} - u\|_{L^{q+1}(\Omega_T)} \xrightarrow{i\to \infty} 0.
$$
With these choices, let us denote $\mollifytime{u}{i} := \mollifytime{u}{h_i,(\psi_o)_{\delta_i}}$ and $\mollifytime{\psi}{i}:= \mollifytime{\psi}{h_i,(\psi_o)_{\delta_i}}$.

At this point, suppose that $k > \sup_{Q \cap \Omega_T} \psi$.
We define a comparison map
$$
v_i = \mollifytime{u}{i} - \eps \varphi \frac{\left( \mollifytime{u}{i} - k \right)_+}{\left( \mollifytime{u}{i} - k \right)_+ + \sigma} + \| \psi - \mollifytime{\psi}{i} \|_{L^\infty(\Omega_T)},
$$
where $\varphi \in C_0^\infty(Q\cap \Omega_T, \R_{\geq 0})$
and $\eps \leq \frac{1}{2 \|\varphi \|_\infty} ( k - \sup_{Q \cap \Omega_T} \psi )$.
Then, $v_i$ with large enough $i$ is an admissible comparison map. Indeed, 
if $\mollifytime{u}{i} \leq k$, then
$$
v_i \geq \mollifytime{u}{i} - \mollifytime{\psi}{i} + \psi \geq \psi \quad \text{ in } \Omega_T.
$$
Further, if $\mollifytime{u}{i} > k$, note that by our choice of $\varphi$ we have that
$$
v_i \geq \mollifytime{u}{i} - \mollifytime{\psi}{i} + \psi \geq \psi \quad \text{ in } \Omega_T \setminus Q.
$$ 
Moreover, by choosing $i$ large enough, there holds $\|\psi - \mollifytime{\psi}{i}\|_{L^\infty(\Omega_T)} \leq \tfrac{1}{2}(k - \sup_{Q\cap \Omega_T} \psi )$. Now for all such $i$ we have
\begin{align*}
v_i &\geq \mollifytime{u}{i} - \eps \|\varphi\|_\infty - \| \psi - \mollifytime{\psi}{i} \|_\infty \\
&\geq k - \tfrac{1}{2}(k - \sup_{Q \cap \Omega_T} \psi) - \tfrac12(k - \sup_{Q \cap \Omega_T} \psi) \\
&= \sup_{Q \cap \Omega_T} \psi \geq \psi
\end{align*}
in $\Omega_T \cap Q$, so that $v_i \geq \psi$ in $\Omega_T$ in  every case.
Furthermore, due to the monotonicity of $\R \ni v \mapsto \boldsymbol{v}^q$, for any $\eps \leq \frac{\sigma}{\|\varphi\|_\infty}$ we have that
\begin{align} \label{e.pos_derivative}
	(\u^q &- \boldsymbol{\mollifytime{u}{i}^q})
	\left[ \partial_t \left( \mollifytime{u}{i} - \eps \varphi \tfrac{\left( \mollifytime{u}{i} - k \right)_+}{\left( \mollifytime{u}{i} - k \right)_+ + \sigma} \right)
	+ \varepsilon \partial_t \varphi \tfrac{\left( \mollifytime{u}{i} - k \right)_+}{\left( \mollifytime{u}{i} - k \right)_+ + \sigma} \right] \nonumber \\
	&= \tfrac{1}{h_i} (\u^q - \boldsymbol{\mollifytime{u}{i}^q}) (u- \mollifytime{u}{i} ) \left(1- \eps \varphi \tfrac{\sigma}{ \left[\left( \mollifytime{u}{i} - k \right)_+ + \sigma \right]^2} \chi_{\{\mollifytime{u}{i} \geq k\}}\right) \geq 0.
\end{align}
Fix $\sigma>0$ and $\eps = \frac{1}{\|\varphi \|_\infty} \min \left \{\frac{1}{2} ( k - \sup_{Q \cap \Omega_T} \psi ), \sigma \right\}$ and choose compactly supported $\alpha$ and $\eta$ in $Q_{R,S} \cap \Omega_T$ such that $\alpha \eta \equiv 1$ in $\spt(\varphi)$.
Using this together with~\eqref{e.pos_derivative}, the fact that $\boldsymbol{\mollifytime{u}{i}^q} \partial_t \mollifytime{u}{i} = \frac{1}{1+q} |\mollifytime{u}{i}|^{q+1}$ and integration by parts, we derive the following estimate for the parabolic part of \eqref{eq:obstacle-wsol}
\begin{align*}
\llangle &\partial_t \u^q , \alpha\eta (v_i-u) \rrangle \\
&= \iint_{\Omega_T} \alpha'\eta \left[ \tfrac{q}{q+1}|u|^{q+1} - \u^q \mollifytime{u}{i} \right] - \alpha \eta \boldsymbol{\mollifytime{u}{i}^q} \partial_t \mollifytime{u}{i}\, \d x \d t \\
&\phantom{+}+ \eps \iint_{\Omega_T}  \alpha' \eta \varphi \u^q \tfrac{\left( \mollifytime{u}{i} - k \right)_+}{\left( \mollifytime{u}{i} - k \right)_+ + \sigma}  + \alpha \eta \varphi \boldsymbol{\mollifytime{u}{i}^q} \partial_t \tfrac{\left( \mollifytime{u}{i} - k \right)_+}{\left( \mollifytime{u}{i} - k \right)_+ + \sigma} \, \d x \d t \\
&\phantom{+}- \iint_{\Omega_T} \alpha' \eta \u^q \|\psi - \mollifytime{\psi}{i}\|_\infty\, \d x \d t + \eps \iint_{\Omega_T}  \alpha \eta \u^q \partial_t \varphi \tfrac{\left( \mollifytime{u}{i} - k \right)_+}{\left( \mollifytime{u}{i} - k \right)_+ + \sigma} \, \d x \d t \\
&\phantom{+} - \iint_{\Omega_T} \alpha\eta  (\u^q - \boldsymbol{\mollifytime{u}{i}^q}) \left(\partial_t \mollifytime{u}{i} - \eps \varphi \partial_t \tfrac{\left( \mollifytime{u}{i} - k \right)_+}{\left( \mollifytime{u}{i} - k \right)_+ + \sigma} \right)\, \d x \d t \\
&\leq \iint_{\Omega_T} \alpha' \eta  \left[ \tfrac{q}{q+1}|u|^{q+1} - \u^q \mollifytime{u}{i} + \tfrac{1}{q+1} |\mollifytime{u}{i}|^{q+1} \right]\, \d x \d t \\
&\phantom{+}+ \eps \iint_{\Omega_T}  \alpha' \eta \varphi \u^q \tfrac{\left( \mollifytime{u}{i} - k \right)_+}{\left( \mollifytime{u}{i} - k \right)_+ + \sigma}  + \alpha \eta \varphi \boldsymbol{\mollifytime{u}{i}^q} \partial_t \tfrac{\left( \mollifytime{u}{i} - k \right)_+}{\left( \mollifytime{u}{i} - k \right)_+ + \sigma} \, \d x \d t \\
&\phantom{+}- \iint_{\Omega_T} \alpha' \eta \u^q \|\psi - \mollifytime{\psi}{i}\|_\infty \, \d x \d t + \eps \iint_{\Omega_T}  \alpha \eta \u^q \partial_t \varphi \tfrac{\left( \mollifytime{u}{i} - k \right)_+}{\left( \mollifytime{u}{i} - k \right)_+ + \sigma} \, \d x \d t \\
&= \iint_{\Omega_T} \alpha' \eta \left[ \tfrac{q}{q+1}|u|^{q+1} - \u^q \mollifytime{u}{i} + \tfrac{1}{q+1} |\mollifytime{u}{i}|^{q+1} \right]\, \d x \d t \\
&\phantom{+}+ \eps \iint_{\Omega_T}  \partial_t \varphi (\u^q - \boldsymbol{\mollifytime{u}{i}^q})\tfrac{\left( \mollifytime{u}{i} - k \right)_+}{\left( \mollifytime{u}{i} - k \right)_+ + \sigma}  \, \d x \d t \\
&\phantom{+} + \eps \iint_{\Omega_T}  \boldsymbol{\mollifytime{u}{i}^q} \partial_t \left(  \varphi \tfrac{\left( \mollifytime{u}{i} - k \right)_+}{\left( \mollifytime{u}{i} - k \right)_+ + \sigma} \right) \, \d x \d t \\
&\phantom{+}- \iint_{\Omega_T} \alpha' \eta  \u^q \|\psi - \mollifytime{\psi}{i}\|_{\infty} \, \d x \d t. 
\end{align*}
By integration by parts, for the term on the third line of the right-hand side of the preceding inequality we obtain that
\begin{align*}
\iint_{\Omega_T}  &\boldsymbol{\mollifytime{u}{i}^q} \partial_t \left(  \varphi \tfrac{\left( \mollifytime{u}{i} - k \right)_+}{\left( \mollifytime{u}{i} - k \right)_+ + \sigma} \right) \, \d x \d t \\
&= - \iint_{\Omega_T} \varphi \partial_t \boldsymbol{\mollifytime{u}{i}^q} \tfrac{\left( \mollifytime{u}{i} - k \right)_+}{\left( \mollifytime{u}{i} - k \right)_+ + \sigma} \, \d x \d t \\
&= - \iint_{\Omega_T}\varphi \partial_t \left( k^q + q\int_k^{\mollifytime{u}{i}} \tfrac{|s|^{q-1}(s-k)_+}{(s-k)_+ + \sigma} \, \d s \right) \, \d x \d t \\
&= \iint_{\Omega_T}\partial_t \varphi \left( k^q + q\int_k^{\mollifytime{u}{i}} \tfrac{|s|^{q-1}(s-k)_+}{(s-k)_+ + \sigma} \, \d s \right) \, \d x \d t. 
\end{align*}
Further, note that the terms on the first, second and fourth line of the right-hand side of the penultimate inequality vanish when $i \to \infty$.
Thus, by combining the estimates and passing to the limit $i \to \infty$, we get that
\begin{align*}
\limsup_{i\to\infty}&\, \llangle \partial_t \u^q, \alpha\eta (v_i-u) \rrangle \\
&\leq \eps \iint_{\Omega_T} \partial_t \varphi \left( k^q + q\int_k^{u} \tfrac{|s|^{q-1}(s-k)_+}{(s-k)_+ + \sigma}\, \d s \right)  \, \d x\d t.
\end{align*}
Since $u_i \to u$ in $L^2(0,T;H^1(\Omega))$ as $i \to \infty$ and $\alpha \eta \equiv 1 $ in $\spt(\varphi)$, for the diffusion part of the variational inequality~\eqref{eq:obstacle-wsol} we compute that
\begin{align*}
\iint_{\Omega_T}&\alpha \mathbf{A}(x,t,\u^q, \nabla u) \cdot \nabla (\eta  (v_i - u) ) \, \d x \d t\\
&= \iint_{\Omega_T}\alpha \mathbf{A}(x,t,\u^q, \nabla u) \cdot \nabla (\eta  (\mollifytime{u}{i} - u) ) \, \d x \d t \\
 &\phantom{+}- \eps \iint_{\Omega_T}\alpha \mathbf{A}(x,t,\u^q, \nabla u) \cdot \nabla \left(\eta \varphi \tfrac{\left( \mollifytime{u}{i} - k \right)_+}{\left( \mollifytime{u}{i} - k \right)_+ + \sigma}\right) \, \d x \d t\\
 &\phantom{+}+ \|\psi - \mollifytime{\psi}{i}\|_\infty \iint_{\Omega_T}\alpha \mathbf{A}(x,t,\u^q, \nabla u) \cdot \nabla  \eta \, \d x \d t \\ 
&\xrightarrow{i\to \infty} - \eps \iint_{\Omega_T} \mathbf{A}(x,t,\u^q, \nabla u) \cdot \nabla \left(\varphi \tfrac{\left( u - k \right)_+}{\left( u - k \right)_+ + \sigma}\right) \, \d x \d t.
\end{align*}
Further, using~\eqref{eq:structure}$_1$, for the term on right-hand side we find that
\begin{align*}
- \eps &\iint_{\Omega_T} \mathbf{A}(x,t,\u^q, \nabla u) \cdot \nabla \left(\varphi \tfrac{\left( u - k \right)_+}{\left( u - k \right)_+ + \sigma}\right) \, \d x \d t \\
&= -\eps \iint_{\Omega_T} \tfrac{\left( u- k \right)_+}{\left( u- k \right)_+ + \sigma} \mathbf{A}(x,t,\u^q, \nabla u) \cdot  \nabla \varphi \, \d x \d t \\
&\phantom{+} - \eps \iint_{\Omega_T} \varphi \tfrac{\sigma}{\left[( u - k )_+ + \sigma\right]^2} \mathbf{A}(x,t,\u^q, \nabla u) \cdot \nabla (u-k)_+ \, \d x \d t \\
&\leq -\eps \iint_{\Omega_T} \tfrac{\left( u- k \right)_+}{\left( u- k \right)_+ + \sigma} \mathbf{A}(x,t,\u^q, \nabla u) \cdot  \nabla \varphi \, \d x \d t.
\end{align*}
By combining all the estimates and dividing by $\eps$ we infer that
\begin{align} \label{eq:obstacle_subsol}
0&\leq \limsup_{i\to \infty} \left( \llangle \partial_t \u^q, \alpha \eta (v_i-u) \rrangle + \iint_{\Omega_T} \alpha \mathbf{A}(x,t,u, \nabla u) \cdot \nabla (\eta (v_i-u)) \, \d x \d t \right) \nonumber \\
&\leq \iint_{\Omega_T} \partial_t \varphi \left( k^q + q\int_k^{u} \tfrac{|s|^{q-1}(s-k)_+}{(s-k)_+ + \sigma}\, \d s \right)  \, \d x\d t \nonumber \\
&\phantom{+} - \iint_{\Omega_T} \tfrac{\left( u- k \right)_+}{\left( u- k \right)_+ + \sigma} \mathbf{A}(x,t,\u^q, \nabla u) \cdot  \nabla \varphi \ \d x \dt.
\end{align}
Thus, finally passing to the limit $\sigma \downarrow 0$, we obtain that
\begin{align*}
\iint_{\Omega_T} -\partial_t \varphi \max\{u,k\}^q + \mathbf{A}(x,t,\max\{u,k\}^q, \nabla \max\{u,k\})  \cdot \nabla \varphi  \, \d x \d t \leq 0. 
\end{align*}
for any $k> \sup_{Q}\psi$.
If $k = \sup_Q \psi$, we use the dominated convergence theorem to
pass to the limit $k \downarrow \sup_{Q}\psi$, which completes the proof of $\eqref{lem:weak_subsol_2}$.

In order to show that $\eqref{lem:weak_subsol_1}$ holds, we define a comparison map
$$
v_i = \mollifytime{u}{i} + \eps \varphi \frac{(\mollifytime{u}{i}-k)_-}{(\mollifytime{u}{i}-k)_- + \sigma} +\| \psi - \mollifytime{\psi}{i}\|_\infty,
$$
where $\mollifytime{\cdot}{i}$ is defined as in the case~$\eqref{lem:weak_subsol_2}$ and $\varphi \in C_0^\infty(Q\cap \Omega_T, \R_{\geq 0})$ together with small enough $\eps > 0$.
Since we have that
$$
v_i \geq \mollifytime{u}{i} + \psi - \mollifytime{\psi}{i} \geq \psi \quad \text{ in } \Omega_T,
$$
$v_i$ is an admissible comparison map for any $k \in \R$.
We let $\eps \leq \frac{\sigma}{\|\varphi\|_{\infty}}$ such that 
\begin{align*} \label{e.pos_derivative_super}
(\u^q &- \boldsymbol{\mollifytime{u}{i}^q})
	\left[\partial_t \left( \mollifytime{u}{i} + \eps \varphi \tfrac{\left( \mollifytime{u}{i} - k \right)_-}{\left( \mollifytime{u}{i} - k \right)_- + \sigma} \right)- \eps \partial_t \varphi \tfrac{\left( \mollifytime{u}{i} - k \right)_-}{\left( \mollifytime{u}{i} - k \right)_- + \sigma} \right] \nonumber \\
&= \tfrac{1}{h_i} (\u^q - \boldsymbol{\mollifytime{u}{i}^q}) (u- \mollifytime{u}{i} ) \left(1- \eps \varphi \tfrac{\sigma}{ \left[\left( \mollifytime{u}{i} - k \right)_- + \sigma \right]^2} \chi_{\{\mollifytime{u}{i} \leq k\}}\right) \geq 0.
\end{align*}
Proceeding similar as in the proof of \eqref{lem:weak_subsol_2}, we conclude the proof of \eqref{lem:weak_subsol_1}.
\end{proof}

For a cylinder $Q_{R,S}(x_o,t_o)$ with a vertex $(x_o,t_o)\in S_T$ such that $0<S<t_o$, we define an extension of $\mathbf{A}(x,t,u,\xi)$ by
\[ \widetilde{\mathbf{A}}(x,t,u,\xi) :=
\begin{cases}
\mathbf{A}(x,t,u,\xi) \quad &\text{ in } Q_{R,S} \cap \Omega_T \\
\xi \quad &\text{ in } Q_{R,S} \setminus \Omega_T.
\end{cases}
\]
Observe that $\widetilde{\mathbf{A}}$ is a Carath\'eodory function satisfying~\eqref{eq:structure} with $C_o$ and $C_1$ replaced by $\min\{1,C_o\}$ and $\max\{1,C_1\}$ respectively.

In the following we show that we can extend $\min\{u,k\}$ and $\max\{u,k\}$ by constant $k$ from $Q_{R,S}\cap \Omega_T$ into $Q_{R,S}$ as super- and subsolution, respectively, to the obstacle free porous medium type equation with vector field $\widetilde{\mathbf{A}}$.  

In the proof we will exploit the following Hardy's inequality.

\begin{lemma} \label{lem:hardy}
Let $\Omega$ be a bounded open set in $\R^n$ which satisfies~\eqref{geometry} and $u \in H^1_0(\Omega)$. Then, there exists a constant $c$ depending only on $n$ and $\alpha_*$ such that
$$
\int_\Omega \left(\frac{|u(x)|}{\dist(x,\partial \Omega)} \right)^2 \, \d x \leq c \int_\Omega |\nabla u(x)|^2 \, \d x.
$$
\end{lemma}

\begin{remark} \label{rem:hardy}
Observe that the positive geometric density condition is not the weakest possible assumption in Lemma~\ref{lem:hardy}. The result holds true under the assumption that $\R^n \setminus\Omega$ is uniformly $2$-thick, see e.g.~\cite{Lewis}. Thus, Lemma~\ref{lem:super_sub_extension} is also valid under the same assumption.
\end{remark}

\begin{lemma} \label{lem:super_sub_extension}
Suppose that $\Omega \subset \R^n$ is a bounded open set and satisfies \eqref{geometry}. Let $(x_o,t_o) \in S_T$ and consider the cylinder $Q_{R,S}(x_o,t_o)$ with $R>0$ and $0<S<t_o$.
Let $u$ be a weak solution to the obstacle problem according to Definition~\ref{d.obstacle-wsol} with obstacle $\psi$ satisfying~\eqref{eq:psi_conds} and lateral boundary values $g$ satisfying~\eqref{eq:g-conds}.
\begin{enumerate}
\item If $k \leq \inf_{S_T \cap Q_{R,S}} g$, then
$$
	u_k :=
	\left\{
	\begin{array}{ll}
		\min \{ u,k \}
		& \text{in } \; Q_{R,S} \cap \Omega_T, \\[5pt]
		k
		& \text{in } \; Q_{R,S} \setminus \Omega_T.
	\end{array}
	\right.
$$
is a weak supersolution in $Q_{R,S}$ in the sense of Definition~\ref{def:localsol_pme} with $\mathbf{A}$ replaced by the vector field $\widetilde{\mathbf{A}}$.
\item If $k \geq \max \{ \sup_{Q_{R,S} \cap \Omega_T} \psi, \sup_{Q_{R,S} \cap S_T} g \}$, then 
$$
	u_k =
	\left\{
	\begin{array}{ll}
		\max \{ u,k \}
		& \text{in } \; Q_{R,S} \cap \Omega_T, \\[5pt]
		k
		& \text{in } \; Q_{R,S} \setminus \Omega_T.
	\end{array}
	\right.
$$
is a weak subsolution in $Q_{R,S}$ in the sense of Definition~\ref{def:localsol_pme} with $\mathbf{A}$ replaced by the vector field $\widetilde{\mathbf{A}}$.
\end{enumerate}
\end{lemma}

\begin{proof}
We start with the case $(2)$. For the proof of the property that $u_k(\cdot,t) \in H^1(B_R(x_o))$ for a.e. $t \in (t_o-S,t_o)$, see e.g.~\cite[Lemma 2.1]{GLL}. We show that we can use $\varphi \in C_0^\infty(Q_{R,S}, \R_{\geq0})$ in Lemma~\ref{lem:weak_subsol}, which does not necessarily vanish on the boundary of $\Omega$. As a starting point, we take inequality~\eqref{eq:obstacle_subsol}, which reads as
\begin{align*} 
0&\leq \limsup_{i\to \infty} \left( \llangle \partial_t \u^q, \alpha \eta (v_i-u) \rrangle + \iint_{\Omega_T} \alpha \mathbf{A}(x,t,\u^q, \nabla u) \cdot \nabla( \eta (v_i-u)) \, \d x \d t \right) \nonumber \\
&\leq \iint_{\Omega_T} \partial_t \varphi \left( k^q + q\int_k^{u} \tfrac{|s|^{q-1}(s-k)_+}{(s-k)_+ + \sigma}\, \d s \right)  \, \d x\d t \\
&\phantom{+} - \iint_{\Omega_T} \tfrac{\left( u- k \right)_+}{\left( u- k \right)_+ + \sigma} \mathbf{A}(x,t,\u^q, \nabla u) \cdot  \nabla \varphi \, \d x \d t \\
&= \mathrm{I} + \mathrm{II}. 
\end{align*}

Denote the inner parallel set by $\Omega_\lambda := \{x\in \Omega : \dist(x,\partial \Omega) \geq \lambda\}$. Let $\eta_\lambda \in C_0^\infty(\Omega,[0,1])$ such that $\eta_\lambda = 1$ in $\Omega_\lambda$ and $|\nabla \eta_\lambda| \leq \frac{c}{\lambda}$ with a numerical constant $c >0$. Now we plug $\varphi \eta_\lambda$ in the place of $\varphi$, which gives us that
\begin{align*}
\mathrm{II} = - &\iint_{\Omega_T} \eta_\lambda \tfrac{\left( u- k \right)_+}{\left( u- k \right)_+ + \sigma} \mathbf{A}(x,t,\u^q, \nabla u) \cdot  \nabla \varphi \, \d x \d t \\
&- \iint_{\Omega_T}\varphi \tfrac{\left( u- k \right)_+}{\left( u- k \right)_+ + \sigma} \mathbf{A}(x,t,\u^q, \nabla u) \cdot  \nabla \eta_\lambda \, \d x \d t.
\end{align*}
We estimate the second term above by
\begin{align*}
- &\iint_{\Omega_T}\varphi \tfrac{\left( u- k \right)_+}{\left( u- k \right)_+ + \sigma} \mathbf{A}(x,t,\u^q, \nabla u) \cdot  \nabla \eta_\lambda \, \d x \d t \\
&\leq \frac{C_1}{\sigma}\iint_{\Omega_T} \varphi (u-k)_+ |\nabla u| |\nabla \eta_\lambda| \, \d x \d t \\
&\leq \frac{c(C_1)}{\sigma} \left( \iint_{(\Omega \setminus \Omega_\lambda) \times (0,T)}  |\nabla u|^2 \, \d x \d t \right)^\frac{1}{2} \left(\iint_{(\Omega \setminus \Omega_\lambda) \times (0,T)}  \varphi^2 \frac{(u-k)_+^2}{\lambda^2} \, \d x \d t \right)^\frac{1}{2} \\
&\leq \frac{c(C_1)}{\sigma} \left( \iint_{(\Omega \setminus \Omega_\lambda) \times (0,T)} |\nabla u|^2 \, \d x \d t \right)^\frac{1}{2} \left(\iint_{(\Omega \setminus \Omega_\lambda) \times (0,T)} \frac{ [\varphi(u-k)_+]^2}{\dist(x,\partial \Omega)^2} \, \d x \d t \right)^\frac{1}{2} \\
&\leq \frac{c(C_1)}{\sigma} \left( \iint_{(\Omega \setminus \Omega_\lambda) \times (0,T)} |\nabla u|^2 \, \d x \d t \right)^\frac{1}{2} \left( \iint_{\Omega_T} \frac{ [\varphi(u-k)_+]^2}{\dist(x,\partial \Omega )^2} \, \d x \d t \right)^\frac{1}{2} \\
&\leq \frac{c(C_1,n,\alpha_*)}{\sigma} \left( \iint_{(\Omega \setminus \Omega_\lambda) \times (0,T)} |\nabla u|^2 \, \d x \d t \right)^\frac{1}{2} \left(\iint_{\Omega_T } |\nabla (\varphi (u-k)_+ )|^2 \, \d x \d t \right)^\frac{1}{2} \\
&\xrightarrow{\lambda \to 0} 0,
\end{align*}
where we used H\"older's inequality in the third line and Hardy's inequality slice-wise in the sixth line. Observe that Hardy's inequality is applicable, since $k \geq \sup_{Q_{R,S} \cap S_T} g$ and thus $\varphi(\cdot,t) (u(\cdot,t) -k)_+\in H^1_0(\Omega)$ for a.e. $t \in (0,T)$. As $\eta_\lambda \to \chi_\Omega$ pointwise when $\lambda \to 0$, we obtain that $\mathrm{I}$ and the term that is left in $\mathrm{II}$ converge to the corresponding integral over $\Omega_T$.

By passing to the limit $\sigma \downarrow 0$, we obtain that
\begin{align*}
\iint_{Q_{R,S} \cap\Omega_T} - &\partial_t \varphi \max\{u,k\}^q + \mathbf{A}(x,t,\max\{u,k\}^q, \nabla\max\{u,k\})  \cdot \nabla \varphi  \, \d x \d t \leq 0
\end{align*}
for $\varphi \in C_0^\infty(Q_{R,S},\R_{\geq0})$. 
Since $u_k = k$ in $Q_{R,S} \setminus \Omega_T$, there holds
\begin{align*}
\iint_{Q_{R,S} \setminus \Omega_T} - \partial_t \varphi \boldsymbol{k}^q +  \widetilde{\mathbf{A}}(x,t,\boldsymbol{k}^q, \nabla k) \cdot \nabla \varphi   \, \d x \d t = 0.
\end{align*}
By summing up the two (in)equalities above it follows that 
$$
\iint_{Q_{R,S}} -\partial_t \varphi \boldsymbol{u^q_k} + \widetilde{\mathbf{A}}(x,t,\boldsymbol{u^q_k},\nabla u_k) \cdot \nabla \varphi  \, \d x \d t \leq 0
$$
for every $\varphi \in C_0^\infty(Q_{R,S};\R_{\geq 0})$, which completes the proof for case $(2)$. The case $(1)$ can be treated analogously.
\end{proof}

\subsection{Boundedness of solutions to the obstacle problem} \label{subsec:boundedness}
Next, we show that weak solutions to the obstacle problem are bounded if the boundary data are bounded and the obstacle is continuous up to the boundary. First, we state a maximum principle for weak subsolutions to (obstacle free) porous medium type equations.

\begin{lemma} \label{lem:maximum-principle}
Let $u \in L^2(0,T;H^1(\Omega)) \cap L^q(\Omega_T)$ be a weak subsolution to~\eqref{eq:pde}$_1$ and $k \in \R$. If $(u - k)_+(\cdot,t) \in H^1_0(\Omega)$ for a.e. $t \in (0,T)$ and 
\begin{equation} \label{eq:sub-initial-max}
\bint_0^h \int_\Omega (\u^q-\boldsymbol{k}^q)_+ \, \d x \d t \xrightarrow{h \downarrow 0} 0,
\end{equation}
then 
$$
u \leq k \quad \text{ a.e. in } \Omega_T.
$$
\end{lemma}

\begin{proof}
Fix $\eps > 0$ and let $t_1 = \eps /2$ and $t_2 \in (4\eps,T)$. We can write the mollified weak formulation of~\eqref{eq:localsol_pme} for weak subsolutions as
$$
\iint_{\Omega_{t_1,t_2}} \partial_t \mollifytime{\u^q}{h} \varphi +  \mollifytime{\mathbf{A}(x,t,\u^q,\nabla u)}{h} \cdot \nabla \varphi \, \d x \d t \leq \frac 1 h \int_\Omega \u^q(t_1)  \int_{t_1}^{t_2} e^\frac{t_1-s}{h} \varphi(x,s) \, \d s \d x,
$$
where $\mollifytime{\cdot}{h}$ denotes the mollification in~\eqref{e.time_moll_g} with initial values zero. Let us use a test function $\varphi = \alpha \frac{(u-k)_+}{(u-k)_+ + \sigma}$, where $\sigma > 0$ and $\alpha = \alpha(t) \leq 1$ is a piecewise affine approximation of $\chi_{(\eps, \tau)}(t)$ with $\tau \in (4\eps,t_2)$. Observe that $\varphi$ is an admissible test function, since $\varphi \leq 1$ and $\varphi \in L^2(0,T;H^1_0(\Omega))$. When passing to the limit $h \to \infty$, the right hand side vanishes.
Further, by structure condition~\eqref{eq:structure}$_1$ for the divergence part on the left-hand side of the preceding inequality we obtain that
$$
\sigma \iint_{\Omega_T}\alpha \mathbf{A}(x,t,\u^q,\nabla u) \cdot \frac{\nabla u \chi_{\{u > k\}}}{[(u - k)_++\sigma]^2} \, \d x \d t  \geq 0.
$$
Moreover, by the fact that the map $s \mapsto \frac{(s - k)_+}{(s - k)_+ + \sigma}$ is increasing, and by integration by parts, we estimate the parabolic part on the left-hand side of the penultimate inequality by
\begin{align*}
\iint_{\Omega_T} \partial_t \mollifytime{\u^q}{h} \varphi  \, \d x \d t &\geq \iint_{\Omega_T} \alpha \partial_t \mollifytime{\u^q}{h} \frac{(\mollifytime{\u^q}{h}^\frac{1}{q} - k)_+}{(\mollifytime{\u^q}{h}^\frac{1}{q} - k)_+ + \sigma} \, \d x \d t \\
&= \iint_{\Omega_T} \alpha \partial_t \int_{\boldsymbol{k}^q}^{\mollifytime{\u^q}{h}} \frac{(\boldsymbol{s}^\frac{1}{q} - k)_+}{(\boldsymbol{s}^\frac{1}{q} - k)_+ + \sigma} \, \d s \, \d x \d t \\
&= q \iint_{\Omega_T} \alpha \partial_t \int_{k}^{\mollifytime{\u^q}{h}^\frac{1}{q}} \frac{|s|^{q-1}(s - k)_+}{(s - k)_+ + \sigma} \, \d s \, \d x \d t  \\
&= -q \iint_{\Omega_T} \alpha' \int_{k}^{\mollifytime{\u^q}{h}^\frac{1}{q}} \frac{|s|^{q-1}(s - k)_+}{(s - k)_+ + \sigma} \, \d s \, \d x \d t  \\
&\xrightarrow{h\to 0}- q \iint_{\Omega_T} \alpha' \int_{k}^{u} \frac{|s|^{q-1}(s - k)_+}{(s - k)_+ + \sigma} \, \d s \, \d x \d t  \\
&\xrightarrow{\sigma\to 0}- \iint_{\Omega_T} \alpha' (\u^q-\boldsymbol{k}^q)_+ \, \d x \d t.
\end{align*}
By collecting all the estimates above, we obtain that
$$
\frac{1}{\eps} \int_{\tau-\eps}^\tau \int_\Omega (\u^q-\boldsymbol{k}^q)_+ \, \d x \d t \leq \frac{2}{2\eps} \int_{0}^{2\eps} \int_\Omega (\u^q-\boldsymbol{k}^q)_+ \, \d x \d t.
$$
By passing to the limit $\eps \to 0$ and using initial condition~\eqref{eq:sub-initial-max}, we have $u(x,\tau) \leq k$ for a.e. $x \in \Omega$. Since this holds for a.e. $\tau \in (0,t_2)$ and $t_2 \in (0,T)$ is arbitrary, we have that $u \leq k$ a.e. in $\Omega_T$.
\end{proof}

\begin{lemma} \label{lem:boundedness}
Let $u$ be a weak solution to the obstacle problem according to Definition~\ref{d.obstacle-wsol} with obstacle $\psi \in C(\overline{\Omega_T})$, lateral boundary values $g \in L^2(0,T;H^1(\Omega)) \cap L^\infty(\Omega_T)$ with $g \geq \psi$ a.e. in $\Omega_T$ and initial boundary values $g_o \in L^\infty(\Omega)$ with $g_o \geq \psi(\cdot,0)$ a.e. on $\Omega$. Then $u \in L^\infty(\Omega_T)$.
\end{lemma}

\begin{proof}
Observe that $u$ is bounded from below, since $u \geq \psi$ and $\psi \in C(\overline{\Omega_T})$. Let $k := \max \{ \sup_{\Omega_T} g,\sup_{\Omega} g_o\}$. Lemma~\ref{lem:weak_subsol} implies that $u_k := \max\{u,k\}$ is a weak subsolution to~\eqref{eq:pde}$_1$ in $\Omega_T$, since $k \geq \sup_{\Omega_T} \psi$. Furthermore, since $(u_k-k)_+(\cdot,t) = (u-k)_+(\cdot,t) \leq (u-g)_+ (\cdot,t) \in H^1_0(\Omega)$ for a.e. $t \in (0,T)$, it follows that $(u_k - k)_+ (\cdot,t) \in H^1_0(\Omega)$ for a.e. $t \in (0,T)$. Furthermore, $(\u_k^q-\boldsymbol{k}^q)_+ = (\u^q-\boldsymbol{k}^q)_+ \leq (u-k)_+^q \leq (u-g_o)_+^q$ if $q <1$, such that the condition~\eqref{eq:sub-initial-max} is satisfied for $u_k$ by H\"older's inequality and~\eqref{eq:obstacle-initial-q+1}. If $q > 1$ we have 
\begin{align*}
(\u_k^q-\boldsymbol{k}^q)_+ \leq (\u^q-\boldsymbol{g_o}^q)_+ &\leq \left( |u| + |g_o| \right)^{q-1} (u- g_o)_+ \\
&\leq \left( |u - g_o| + 2|g_o| \right)^{q-1} |u- g_o| \\
&\leq c(q) \left( |u - g_o|^q + |g_o|^{q-1} |u-g_o| \right), 
\end{align*}
such that~\eqref{eq:sub-initial-max} holds by using H\"older's inequality and~\eqref{eq:obstacle-initial-q+1}. Now Lemma~\ref{lem:maximum-principle} implies that $u_k \leq k$ a.e.~in $\Omega_T$. Thus, $u \leq u_k = k = \max \{ \sup_{\Omega_T} g, \sup_{\Omega} g_o\}$ a.e.~in $\Omega_T$, completing the proof.
\end{proof}

\section{Auxiliary results} \label{sec:aux}

In this section, we collect further tools that we will need in the proof of Theorem~\ref{thm:main_theorem}.
For $w,k \in \R$ we define 
\begin{align*}
\mathfrak{g}_\pm(w,k) := \pm q \int_k^w  |s|^{q-1} (s-k)_\pm \, \d s.
\end{align*}

The following estimates follow from the definition above, see e.g.~\cite[Lemma 2.2]{BDL}.
\begin{lemma}\label{lem:calg}
There exists a constant $c = c(q) > 0$ such that for all $w,k \in \R$ and $q > 0$, the inequality
$$
\tfrac{1}{c} \left( |w| + |k| \right)^{q-1} (w-k)^2_\pm \leq \mathfrak g_\pm(w,k) \leq c \left( |w| + |k| \right)^{q-1} (w-k)^2_\pm
$$
holds true.
\end{lemma}

\subsection{Tools for local sub(super)solutions} \label{subsec:local-case}

In this section, assume that $u$ is a local weak sub(super)solution to \eqref{eq:pde}$_1$ in some cylinder
$$
	\mathcal{Q} = Q_{R,S}(\hat{x}_o,\hat{t}_o) \subset \R^{n+1}.
$$
We start with the standard energy estimate. The following statement is retrieved as a special case of \cite[Proposition 2.1]{part2} by setting $p=2$. For a detailed proof see also \cite[Proposition 3.1]{BDL}.
\begin{lemma} \label{l.caccioppoli}
Let $z_o = (x_o,t_o) \in \mathcal{Q}$, $Q_{\rho,s}(z_o) \Subset \mathcal{Q}$ and $k \in \R$ be an arbitrary level.
Further, assume that $\varphi$ is a non-negative, piecewise smooth cutoff function vanishing on $\partial B_\rho(x_o) \times (t_o-s,t_o)$.
\begin{enumerate}
\item If $u$ is a local weak subsolution in $\mathcal{Q}$ to \eqref{eq:pde}$_1$ according to Definition~\ref{def:localsol_pme}, then there exists $c = c(C_o,C_1) > 0$ such that
\begin{align*}
	\max \bigg\{
	&\esssup_{t_o -s < t < t_o} \int_{B_\rho(x_o) \times \{t\}} \varphi^2 \mathfrak{g}_{+}(u,k) \, \d x,
	\tfrac{C_o}{2} \iint_{Q_{\rho,s}(z_o)} \varphi^2 |\nabla (u-k)_{+}|^2 \, \d x \d t
	\bigg\} \\
	&\leq
	c \iint_{Q_{\rho,s}(z_o)} [(u-k)_{+}^2|\nabla \varphi|^2 + \mathfrak{g}_{+} (u,k) |\partial_t \varphi^2|] \, \d x \d t \\
	&\phantom{=} + \int_{B_\rho(x_o) \times \{t_o -s\}} \varphi^2 \mathfrak{g}_{+}(u,k)\, \d x.
\end{align*}
\item If $u$ is a local weak supersolution in $\mathcal{Q}$ to \eqref{eq:pde}$_1$ according to Definition~\ref{def:localsol_pme}, then there exists $c = c(C_o,C_1) > 0$ such that
\begin{align*}
	\max \bigg\{
	&\esssup_{t_o -s < t < t_o} \int_{B_\rho (x_o) \times \{t\}} \varphi^2 \mathfrak{g}_{-}(u,k)\, \d x,
	\tfrac{C_o}{2} \iint_{Q_{\rho,s}(z_o)} \varphi^2 |\nabla (u-k)_{-}|^2 \, \d x \d t
	\bigg\} \\
	&\leq
	c \iint_{Q_{\rho,s}(z_o)} [(u-k)_{-}^2|\nabla \varphi|^2 + \mathfrak{g}_{-} (u,k) |\partial_t \varphi^2|] \, \d x \d t \\
	&\phantom{=} + \int_{B_\rho(x_o) \times \{t_o -s\}} \varphi^2 \mathfrak{g}_{-}(u,k) \, \d x.
\end{align*}
\end{enumerate}
\end{lemma}

Throughout Section~\ref{subsec:local-case}, we will use parameters $\muminus, \muplus \in \R$ and $\bomega > 0$ satisfying
$$
\muminus \leq \inf_{\mathcal{Q}} u,\quad \muplus \geq \sup_{\mathcal{Q}} u \quad \text{and}\quad \bomega \geq \muplus - \muminus. 
$$

\subsubsection{Tools for the case near zero}
First we state a shrinking lemma in the degenerate case. The proof is an adaptation of~\cite[Lemma 4.2]{BDL} to the porous medium setting.

\begin{lemma}\label{lem:shrinking-deg}
Let $0< q < 1$, $j_* \in \N$ be an arbitrary positive integer and $\eps \in (0,1)$. Denote $\theta = (2^{-j_*}\eps \bomega)^{q-1}$ and suppose that $Q_{2\rho,\theta \rho^2}(x_o, t_o) \Subset \mathcal{Q}$. Let $u$ be a locally bounded local weak sub(super)solution to \eqref{eq:pde}$_1$ in $\mathcal{Q}$. Suppose that for some $\alpha \in(0,1)$
$$
\left|\left\{\pm(\mupm - u(\cdot,t)) \geq \varepsilon \bomega\right\} \cap B_\rho(x_o) \right| \geq  \alpha |B_\rho(x_o)|
$$
holds for all $t \in (t_o - \theta \rho^2,t_o]$. 
Then there exists $c = c(C_o,C_1,q,n,\alpha) >0$ such that 
$$
\left|\left\{\pm(\mupm - u) \leq \frac{\varepsilon \bomega}{2^{j_*}}\right\} \cap Q_{\rho,\theta\varrho^2}(x_o,t_o) \right| \leq \frac{c }{\sqrt{ j_*}}|Q_{\rho,\theta\varrho^2}(x_o,t_o)|.
$$
\end{lemma}

Then we state a De Giorgi type lemma in the local, obstacle free case. See~\cite[Lemma 3.1]{part2} in connection with \cite[Lemma 2.2]{Liao-lsc}.

\begin{lemma} \label{lem:de-giorgi-deg}
Let $q > 0$ and $u$ be a locally bounded local weak sub(super)solution to \eqref{eq:pde}$_1$ in $\mathcal{Q}$. Let $\theta = \left(\xi \bomega \right)^{q-1}$ for some $\xi \in (0,1)$ and $Q_{\rho,\theta\varrho^2}(x_o,t_o) \Subset \mathcal{Q}$. There exists a constant $\nu \in (0,1)$ depending only on $C_o,C_1,n$ and $q$ such that if
$$
\left| \left\{ \pm (\mupm - u) \leq \xi \bomega \right\}\cap Q_{\rho,\theta\varrho^2}(x_o,t_o) \right| \leq \nu \left|Q_{\rho,\theta\varrho^2}(x_o,t_o) \right|,
$$
then we have that
$$
\left| \mupm \right| > 8 \xi \bomega,
$$
or
$$
\pm (\mupm - u) \geq \tfrac12 \xi \bomega \quad \text{ a.e. in } Q_{\frac{\varrho}{2},\theta\left(\frac{\varrho}{2}\right)^2}(x_o,t_o).
$$
\end{lemma}

In the singular case, we exploit the result on expansion of positivity in~\cite[Proposition 2.2]{Liao}.
Since it is stated and proved for a different formulation of \eqref{eq:pde}$_1$, which is equivalent in the case of locally bounded solutions, we have to add the assumption that $\max \{ |\muminus|,|\muplus|\} \leq 16 M$ in the following proposition. This is due to the fact that the expansion of positivity in~\cite[Proposition 2.2]{Liao} is shown for the quantity $\pm ((\mupm)^q - \u^q)$ in our notation, but we use the result for $\pm (\mupm - u)$.

\begin{proposition}
\label{prop:expansion-of-positivity-q}
Let $q > 1$ and assume that $u$ is a locally bounded local weak sub(super)solution
to \eqref{eq:pde}$_1$ according to Definition~\ref{def:localsol_pme} in a cylinder $\mathcal{Q} \supset B_{16 \rho}(y)\times (s, s + \delta M^{q-1} \rho^2]$ for some $M>0$ and $\delta \in(0,1)$ determined below.
Suppose that
$$
	\big| \big\{ \pm \big( \mupm - u(\cdot,s) \big) \geq M \big\}
	\cap B_\varrho(y) \big|
	\geq
	\alpha |B_\varrho(y)|
$$
for $\alpha \in (0,1)$, and that $\max\{|\muplus|,|\muminus|\} \leq 16 M$.
Then, there exist constants $\xi, \eta, \varepsilon, \delta \in (0,1)$ depending only on $C_o,C_1,n,q$ and $\alpha$ such that
$$
	\big| \mupm \big| > 8 \xi M
$$
or the inequality
$$
	\pm \big( \mupm - u(\cdot,t) \big) \geq \eta M
	\quad \text{a.e.~in } B_{2\varrho}(y)
$$
holds true for all times
$$
t \in \Big[ s + (1-\epsilon) \delta M^{q-1} \varrho^2,
		s + \delta M^{q-1} \varrho^2 \Big].
$$
\end{proposition}

\begin{proof}
We only give the proof for a local weak subsolution, since the proof for supersolutions is analogous.
For $v := \u^q$ and $m = \frac{1}{q}$ we have that
$$
\partial_t v - \Div \mathbf{A}\big(x,t,v,\nabla \boldsymbol{v}^m \big) \leq  0.
$$
Since $u$ is bounded in $\Omega_T$, we find that $\nabla v = \frac{1}{m} |v|^{1-m} \nabla \boldsymbol{v}^m \in L^2(\Omega_T,\R^n)$.
Therefore, $\nabla v$ is the weak gradient of $v$.
Further, since $\nabla \boldsymbol{v}^m$ is zero a.e.~in the set $\{v=0\}$, we have that
$$
\nabla \boldsymbol{v}^m = m |v|^{m-1} \nabla v \chi_{\{v \neq 0\}},
$$
a.e.~in $\Omega_T$.
Now, we define the vector field $\widetilde{\mathbf{A}} \colon \Omega_T \times \R \times \R^n \to \R^n$ by
$$
\widetilde{\mathbf{A}}(x,t,v, \xi) = \mathbf{A} \big(x,t,v,m |v|^{m-1} \chi_{\R \setminus \{0\}}(v) \xi \big).
$$ 
Then $\widetilde{\mathbf{A}}$ satisfies the structure conditions
$$
	\left\{
	\begin{array}{l}
	|\widetilde{\mathbf{A}}(x,t,v, \xi)| \leq m C_1 |v|^{m-1} \chi_{\R \setminus \{0\}}(v) |\xi|, \\[5pt]
	\widetilde{\mathbf{A}}(x,t,v, \xi) \cdot \xi \ge m C_o |v|^{m-1} \chi_{\R \setminus \{0\}}(v) |\xi|^2
	\end{array}
	\right.
$$
for a.e. $(x,t) \in \Omega_T$ and every $(v,\xi) \in \R \times \R^n$.
In this setting, we are able to apply the result on expansion of positivity in the singular case in~\cite{Liao}.
Defining $\muplus_v := (\muplus)^\frac{1}{m} \geq \sup_{\mathcal{Q}} v$, we rewrite the assumption that
$$
\big| \big\{ \muplus - u(\cdot,s) \geq M \big\}
	\cap B_\varrho(y) \big| 
	\geq
	\alpha |B_\varrho(y)|
$$
as
$$
\big| \big\{ (\muplus_v)^m - \boldsymbol{v}^m(\cdot,s) \geq M \big\}
	\cap B_\varrho(y) \big| 
	\geq
	\alpha |B_\varrho(y)|.
$$
Since there exists a constant $c_1 = c_1(m) >0$ such that $c_1 |\boldsymbol{a}^m - \boldsymbol{b}^m| \leq |a-b|^m$ for any $a,b \in \R$, this implies that
$$
\big| \big\{ \muplus_v - v(\cdot,s) \geq (c_1 M)^\frac{1}{m} \big\}
	\cap B_\varrho(y) \big| 
	\geq
	\alpha |B_\varrho(y)|.
$$
Then, by using \cite[Proposition 2.2]{Liao} with $M$ replaced by $c_1M$, we infer that there exist constants $\tilde{\xi}, \varepsilon, \tilde{\delta}, \tilde{\eta} \in (0,1)$ depending only on $C_o,C_1,n,m$ and $\alpha$ such that
$$
	|\muplus_v| > \tilde{\xi} (c_1 M)^\frac{1}{m}
$$
or
$$
	\muplus_v - v(\cdot,t) \geq \tilde{\eta} (c_1 M)^\frac{1}{m} \quad \text{ a.e. in } B_{2 \rho}(y)
$$
for all $t \in \big[s+  (1-\eps)\tilde{\delta} (c_1 M)^{\frac{1}{m} -1} \rho^2, s+ \tilde{\delta} (c_1 M)^{\frac{1}{m} -1} \rho^2 \big]$.
At this point, we use the fact that $0<m<1$ and the assumption that $\max\{|\muplus_v|, |v|\} \leq \max\{|\muplus|, |\muminus| \}^\frac{1}{m} \leq (16M)^\frac{1}{m}$ to conclude that there exists a constant $c_2 = c_2(m) > 0$ such that
$$
	(\muplus_v)^m - \boldsymbol{v}^m
	\geq
	c_2 (|\muplus_v| + |v|)^{m-1} (\muplus_v -  v)
	\geq
	c_2 (32M)^{1-\frac{1}{m}} (\muplus_v - v).
$$
Combining the preceding inequalities, reverting to $u = \boldsymbol{v}^m$ and defining $\xi := \frac18 c_1 \tilde{\xi}^m$, $\eta := 2^{5\left(1-\frac{1}{m}\right)} c_1^\frac{1}{m} c_2 \tilde{\eta}$ and $\delta := c_1^{\frac{1}{m}-1} \tilde{\delta}$, this translates to the conclusion that
$$
	|\muplus| > 8 \xi M
$$
or
$$
	\muplus- u(\cdot,t) \geq \eta M \quad \text{ a.e. in } B_{2 \rho}(y)
$$
holds for every $t \in \big[s+  (1-\eps)\delta M^{q-1} \rho^2, s+ \delta M^{q-1} \rho^2\big]$.
\end{proof}

\begin{remark}
\label{rem:expansion}
Replacing $c_1 M$ by $\kappa c_1 M$ with $\kappa \in [\kappa_o,1]$ for some $\kappa_o \in (0,1)$ in the proof of Proposition~\ref{prop:expansion-of-positivity-q}, we obtain that under the hypotheses of Proposition \ref{prop:expansion-of-positivity-q} there holds $\big|\boldsymbol{\mu}^\pm\big| > 8\xi\kappa_o M$ or
$$
	\pm \big( \mupm - u(\cdot,t) \big) \geq \eta \kappa_o M
	\quad \text{a.e.~in } B_{2\varrho}(y)
$$
for all times
$$
t \in \Big[ s + \kappa_o^{q-1} (1-\epsilon) \delta M^{q-1} \varrho^2,
		s + \delta M^{q-1} \varrho^2 \Big].
$$
Thus, redefining $\xi$ as $\xi \kappa_o$ and $\eta$ as $\eta \kappa_o$, pointwise positivity can be claimed as close to $s$ as we need. 
\end{remark}

The next result is an analogue to~\cite[Lemma 3.3]{part2}.
The only difference in the proof is that we replace~\cite[Proposition 4.1, Chapter 4]{DGV}, i.e.~expansion of positivity for $p>2$, by the simpler result \cite[Proposition 2.1, Chapter 4]{DGV} for $p=2$.

\begin{lemma}
\label{lem:mu trapped}
Let $q>0$ and introduce the parameters $\Lambda, c > 0$ and $\alpha \in (0,1)$. Assume that $u$ is a locally bounded local weak sub(super)solution in $\mathcal{Q}$
to \eqref{eq:pde}$_1$.
Suppose that
$$
	c \bomega \leq \pm \mupm \leq \Lambda \bomega
$$
and for some $0 < a \leq \frac{1}{2}c$ there holds
$$
	\big| \big\{ \pm \big( \mupm - u(\cdot,s) \big) \geq a \bomega \big\}
	\cap B_\varrho \big|
	\geq
	\alpha |B_\varrho|.
$$
Then, there exist constants $\eta, b \in (0,1)$
depending only on $C_o,C_1,n,q$, $\Lambda, c, a$ and $\alpha$ such that
$$
	\pm \big( \mupm - u) \geq \eta \bomega
	\quad \text{a.e.~in }
	B_{2 \varrho} \times
	\Big( s + \tfrac12 b \bomega^{q-1} \varrho^2,
	s + b \bomega^{q-1} \varrho^2 \Big],
$$
provided that $B_{8 \rho} \times (s-b \bomega^{q-1} \varrho^2,s+ b \bomega^{q-1} \varrho^2 ) \Subset \mathcal{Q}$.
\end{lemma}

\subsubsection{Tools for the case away from zero}

Next we will state and prove tools for the obstacle free problem in the case where $u$ is bounded away from zero.
This is expressed in terms of $\mupm$ and $\bomega$ by assuming that
\begin{equation} \label{eq:above-zero-dg-local}
\muminus > \xi \bomega \quad \text{ and}\quad \tfrac12 \muplus \leq \theta^\frac{1}{q-1} \leq \tfrac{\xi + 1}{\xi} \muplus
\end{equation}
or
\begin{equation} \label{eq:below-zero-dg-local}
\muplus <- \xi \bomega \quad \text{ and}\quad \tfrac12 |\muminus| \leq \theta^\frac{1}{q-1} \leq \tfrac{\xi+1}{\xi} |\muminus|
\end{equation}
holds true for some constant $\xi \in (0,1)$.

First, we state an auxiliary lemma concerning the truncation levels used in this subsection.

\begin{lemma} \label{lem:away-zero-levels}
Let $\eta \in \big(0,\frac{\xi}{2}\big]$ and $k_\pm = \mupm \mp \eta \bomega$. If~\eqref{eq:above-zero-dg-local} holds, then
$$
\tfrac{\xi}{\xi + 1}\muplus < \muminus \leq \muplus,
$$
$$
\muminus \leq k_- < \tfrac{3}{2} \muminus,\quad \text{and }\quad \tfrac12 \muplus < k_+ \leq \muplus.
$$
Furthermore, $\theta^\frac{1}{q-1} \approx k_\pm$ up to a constant depending on $\xi$.

If~\eqref{eq:below-zero-dg-local} holds, then
$$
 \muminus \leq \muplus < \tfrac{\xi}{\xi +1} \muminus,
$$
$$
\muminus \leq k_- < \tfrac12 \muminus,\quad \text{and } \quad \tfrac32 \muplus < k_+ \leq \muplus.
$$
Furthermore, $\theta^\frac{1}{q-1} \approx |k_\pm|$ up to a constant depending on $\xi$.
\end{lemma}

Then we give a version of De Giorgi type lemma.
\begin{lemma} \label{l.de_giorgi2_scheven2}
Let $Q_{\rho,\frac12 \theta \rho^2}(z_o) \Subset \mathcal{Q}$ be a parabolic cylinder such that either~\eqref{eq:above-zero-dg-local} or~\eqref{eq:below-zero-dg-local} holds true. Let $\eta \in \big(0,\frac{\xi}{2}\big]$.
Assume that $u$ is a locally bounded, local weak sub(super)solution to \eqref{eq:pde}$_1$  in the sense of Definition~\ref{def:localsol_pme} in $\mathcal{Q}$.
Then there exists a constant $\nu_1 = \nu_1(n,q,C_o,C_1,\xi) \in (0,1)$ such that if
$$
\left| \left\{ \pm (\boldsymbol{\mu}^\pm - u )  < \eta \boldsymbol{\omega} \right\}
\cap Q_{\rho,\frac12 \theta \rho^2}(z_o) \right|< \nu_1 \left| Q_{\rho,\frac12 \theta \rho^2}(z_o) \right|, 
$$
then
$$
 \pm( \boldsymbol{\mu}^\pm - u)  \geq \tfrac12 \eta \boldsymbol{\omega}\quad \text{ a.e. in } Q_{\frac{\rho}{2},\frac12 \theta \left(\frac{\rho}{2}\right)^2} (z_o)
$$
\end{lemma}

\begin{proof}
We prove the result for supersolutions in case~\eqref{eq:above-zero-dg-local}, where $u$ is above zero.
The considerations for the remaining three cases, i.e.~for subsolutions in case~\eqref{eq:above-zero-dg-local} and for super- and subsolutions in case~\eqref{eq:below-zero-dg-local} are analogous.
We omit $z_o$ for simplicity and start the proof by defining
$$
	\begin{array}{c}
		k_j = \boldsymbol{\mu}^- + \Big( \frac{\eta}{2} + \frac{\eta}{2^{j+1}} \Big) \boldsymbol{\omega},\,
		\rho_j = \frac{\rho}{2} + \frac{\rho}{2^{j+1}},\,
		B_j = B_{\rho_j},\,
		Q_j := Q_{\rho_j, \frac12 \theta \rho_j^2}.
	\end{array}
$$
Observe that clearly $k_j > 0$ for all $j \in \N_0$.
Let $0 \leq \varphi \leq 1$ be a cut-off function that equals $1$ in $Q_{j+1}$ and vanishes on the parabolic boundary of $Q_j$ such that
$$
\left| \nabla \varphi \right| \leq c\frac{2^j}{\rho}\quad \text{ and }\quad \left| \partial_t \varphi \right|\leq c \frac{2^{2j}}{\theta \rho^2}.
$$
From the fact that $0 < \boldsymbol{\mu}^- \leq u < k_j$ in the set $A_j := \{ u < k_j \} \cap Q_j$ and the energy estimate in Lemma~\ref{l.caccioppoli} (2), we obtain
\begin{align*}
	\min&\left\{ \left( \boldsymbol{\mu}^- \right)^{q-1}, k_j^{q-1} \right\}
	\esssup_{-\frac12\theta \rho_j^2<t<0} \int_{{B}_j} \varphi^2( u-{k}_j)_-^2\, \d x 
	+ \iint_{{Q}_j} \varphi^2| \nabla (u-{k}_j)_- |^2 \, \d x \d t \\
	&\leq
	\esssup_{-\frac12\theta \rho_j^2<t<0} \int_{{B}_j} \varphi^2 \left( |u| + |k_j| \right)^{q-1} ( u-{k}_j)_-^2\, \d x 
	+ \iint_{{Q}_j} \varphi^2 | \nabla (u-k_j)_- |^2 \, \d x \d t \\
	&\leq
	c \iint_{Q_j} (u-k_j)_-^2 |\nabla \varphi|^2 \, \d x \d t + c \iint_{Q_j} (|u| + |k_j|)^{q-1} (u-k_j)_-^2 |\partial_t \varphi^2| \, \d x \d t \\
	&\leq
	c \frac{2^{2j}}{\rho^2} \left( 1 + \max \left\{ \left(\boldsymbol{\mu}^- \right)^{q-1}, k_j^{q-1} \right\} \theta^{-1} \right) \iint_{Q_j} (u-k_j)_-^2 \, \d x \d t \\
	&\leq 
	c \frac{2^{2j} \left( \eta \boldsymbol{\omega}\right)^2}{\rho^2} \left( 1 + \max \left\{ \left( \boldsymbol{\mu}^- \right)^{q-1}, k_j^{q-1} \right\} \theta^{-1} \right) \left| A_j\right| \\
	&\leq c \frac{2^{2j} \left( \eta \boldsymbol{\omega}\right)^2}{\rho^2} \left| A_j\right|
\end{align*}
with a constant $c = c(C_o,C_1,n,q,\xi)$, where we used Lemma~\ref{lem:away-zero-levels}.
With these estimates at hand, we infer in particular
$$
\esssup_{-\frac12\theta \rho_j^2<t<0} \int_{{B}_j} \varphi^2 ( u-{k}_j)_-^2\, \d x \leq c \frac{2^{2j} \left( \eta \boldsymbol{\omega}\right)^2}{\theta \rho^2} \left| A_j\right|.
$$
Since $k_j - u \geq k_j -  k_{j+1} = 2^{-(j+2)} \eta \boldsymbol{\omega}$ in the set $\{u \leq  k_{j+1}\}$, by H\"older's and Sobolev's inequalities~\cite[Proposition 4.1, Chapter 2]{DGV} we obtain that
\begin{align*}
	\frac{\eta \boldsymbol{\omega}}{2^{j+2}} &\left| A_{j+1} \right|
	\leq
	\iint_{{Q}_j}  (u-  k_j )_- \varphi \, \d x \d t \\
	&\leq
	c \left( \iint_{ Q_j} [ (u- k_j)_- \varphi ]^\frac{2(n+2)}{n} \, \d x \d t \right)^\frac{n}{2(n+2)} \left| A_j \right|^{1- \frac{n}{2(n+2)}} \\
	&\leq
	\left( \iint_{ Q_j} | \nabla [ (u- k_j)_- \varphi ] |^2 \, \d x \d t \right)^\frac{n}{2(n+2)}
	\left( \esssup_{-\frac12\theta  \rho_j^2 < t < 0} \int_{ B_j} \varphi^2(u- k_j)_-^2 \, \d x \right)^\frac{1}{n+2} \\
	&\phantom{=}
	\cdot \left| A_j \right|^{1- \frac{n}{2(n+2)}} \\
	&\leq
	c \left( \frac{2^{2j} (\eta \boldsymbol{\omega})^2 }{\rho^2} \right)^\frac{n}{2(n+2)} \left( \frac{2^{2j} (\eta \boldsymbol{\omega})^2}{\theta \rho^2} \right)^\frac{1}{n+2} \left| A_j \right|^{1+\frac{1}{n+2}} \\
	&=
	c \frac{2^j \eta \boldsymbol{\omega}}{\theta^\frac{1}{n+2} \rho} \left| A_j \right|^{1+ \frac{1}{n+2}}.
\end{align*}
Dividing by $|Q_{j+1}|$ and denoting $Y_j = |A_j| / |Q_j|$, we conclude that
$$
Y_{j+1} \leq c 2^{2j} Y_{j}^{1+ \frac{1}{n+2}}
$$
for a constant $c = c(C_o,C_1,n,q,\xi)$. Setting $\nu_1 \leq c^{-(n+2)} 4^{-(n+2)^2}$,
we conclude the proof by using fast geometric convergence lemma~\cite[Lemma 5.1, Chapter 2]{DGV}.
\end{proof}

In the case where $u$ is away from zero, the next lemma transfers positivity from a measure condition at a single time slice to a pointwise estimate in a whole cylinder. The proof applies Lemma~\ref{l.de_giorgi2_scheven2}.

\begin{lemma} \label{l.supersub-slicewise-to-pointwise}
Let $Q_{\rho, \theta \rho^2}(z_o) \Subset \mathcal{Q}$ be a parabolic cylinder such that either~\eqref{eq:above-zero-dg-local} or~\eqref{eq:below-zero-dg-local} holds true. Assume that $u$ is a locally bounded, local weak sub(super)solution in $\mathcal{Q}$ according to Definition~\ref{def:localsol_pme}. Then for any $\nu \in (0,1)$ there exists a constant $a = a(n,q,C_o,C_1,\nu,\xi) \in \big(0,\frac{1}{64}\big]$ such that if
$$
\left| \left\{ \pm ( \boldsymbol{\mu}^\pm - u(\cdot,t) )  \geq \tfrac38 \bomega \right\} \cap B_{\rho}(x_o) \right| > \nu |B_{\rho}(x_o)|
$$
for every $t \in (t_o-\theta\rho^2,t_o]$, then
$$
\pm (\boldsymbol{\mu}^\pm - u) \geq a \bomega \quad \text{ a.e. in } Q_{\frac{\rho}{2}, \frac12 \theta \left( \frac{\rho}{2} \right)^2}(z_o).
$$
\end{lemma}

\begin{proof}
We prove the result for supersolutions in case~\eqref{eq:above-zero-dg-local}. The remaining three cases are analogous. 
For simplicity, we let $z_o= (0,0)$ and omit it from the notation. Define $Q_2 = B_\rho \times \big(-\tfrac12 \theta \rho^2,0\big]$ and $Q_1 = B_\rho \times \big(-  \theta \rho^2,0\big]$ so that $Q_2 \subset Q_1$. Further, set $k_j= \muminus + \tfrac{\xi}{2^j}\bomega$ for $j \in \N_{\geq3}$ and $A_j = \{u < k_j\}\cap Q_2$. By De Giorgi's isoperimetric inequality~\cite[Lemma 2.2, Chapter 2]{DGV} we obtain
\begin{align*}
	( k_j - k_{j+1} ) &| \{ u(\cdot,t) < k_{j+1} \} \cap B_\rho | \\
	&\leq
	\frac{c(n) \rho^{n+1}}{| \{ u(\cdot,t) > k_j \} \cap B_\rho |}
	\int_{B_\rho \cap \{ k_{j+1} < u(\cdot,t) < k_j \}} | \nabla u | \, \d x  \\
	&\leq
	\frac{c(n) \rho}{\nu} \int_{B_\rho \cap \{ k_{j+1} < u(\cdot,t) < k_j \}} | \nabla u | \, \d x.
\end{align*}
Integrating over $\big(- \tfrac12 \theta \rho^2, 0\big)$, we find that
\begin{align*}
( k_j - k_{j+1} ) |A_{j+1}| &\leq \frac{c(n) \rho}{\nu} \iint_{A_j \setminus A_{j+1}} \left| \nabla u \right| \, \d x \d t \\
&\leq \frac{c(n) \rho}{\nu} \left| A_j \setminus A_{j+1} \right|^\frac12 \left( \iint_{A_j \setminus A_{j+1}} |\nabla u|^2 \, \d x \right)^\frac12 \\
&\leq \frac{c(n) \rho}{\nu} \left| A_j \setminus A_{j+1} \right|^\frac12 \left( \iint_{Q_2} \left| \nabla (u-k_j)_- \right|^2 \, \d x \d t \right)^\frac12.
\end{align*}
By applying Lemma~\ref{l.caccioppoli} (2) we get
\begin{align*}
\iint_{Q_2} | \nabla (u-k_j)_-|^2 \, \dx \d t &\leq c \left( \frac{1}{\rho^2} + \frac{\max \{k_j^{q-1}, \left(\boldsymbol{\mu}^-\right)^{q-1}\}}{ \theta \rho^2} \right) \iint_{Q_1} (u-k_j)_-^2 \, \d x \d t \\
&\leq \frac{c}{ \rho^2} \left( \frac{\boldsymbol{\omega}}{2^j} \right)^2 |Q_1|,
\end{align*}
for $c = c(C_o,C_1,q,\xi) > 0$ by using Lemma~\ref{lem:away-zero-levels}. 
By combining the two estimates above and using $k_j - k_{j+1} = \xi 2^{-(j+1)} \bomega$ we have
$$
|A_{j+1}|^2 \leq \frac{c}{\nu^2} |A_j \setminus A_{j+1}| |Q_1|.
$$
By summing this over $j= 3, ..., s_o + 1$ for some $s_o \in \N_{\geq 3}$ we obtain
$$
s_o |A_{s_o+2}|^2 \leq \frac{c}{\nu^2}|Q_1|^2 \leq \frac{c}{\nu^2}|Q_2|^2
$$
for $c = c(C_o,C_1,n,q,\xi) >0$. By choosing $s_o$ large enough, we have
$$
\left| \left\{ u < \muminus + \tfrac{\xi}{2^{s_o+2}}\bomega \right\} \cap Q_{\rho, \frac12 \theta \rho^2} \right| < \nu_1 |Q_{\rho, \frac12 \theta \rho^2}|.
$$
At this stage we apply Lemma~\ref{l.de_giorgi2_scheven2} to conclude that 
\begin{align*}
u &\geq \muminus + \tfrac{\xi}{2^{s_o+3}} \bomega\quad \text{ a.e. in } Q_{\frac{\rho}{2}, \frac12 \theta \left(\frac{\rho}{2} \right)^2}.
\qedhere
\end{align*}
\end{proof}

\subsection{Tools for sub(super)solutions near the initial boundary} \label{subsec:initial-case}

Throughout this section, we are concerned with weak sub(super)solutions to \eqref{eq:pde}$_1$ in a subset $\mathcal{Q}^+ = Q_{R,S}^+(\hat x_o,0) \subset \Omega_T$ attaining an initial datum.
Next, we prove an energy estimate near the initial boundary, see also~\cite[Proposition 2.2]{part2} and \cite[Proposition 3.2]{BDL}.

\begin{lemma} \label{lem:energy-est-initial}
Let $u$ be a weak sub(super)solution to \eqref{eq:pde}$_1$ in $\Omega_T$ with an initial datum $g_o \in L^\infty_{\loc}(\Omega)$. There exists a constant $c = c(C_o,C_1)>0$ such that for all cylinders $Q_{\rho,s}^+(x_o,0) \subset \Omega_T$ and every $k \in \R$ satisfying
\[
\begin{cases}
k\geq \sup_{B_\rho(x_o)} g_o \quad \text{ for weak subsolutions, }\\
k\leq \inf_{B_\rho(x_o)} g_o \quad \text{ for weak supersolutions, }
\end{cases}
\]
there holds
\begin{align*}
\sup_{0<t<s} &\int_{B_\rho(x_o) \times \{t\}} \eta^2 \mathfrak{g}_\pm (u,k)\, \d x + \iint_{Q_{\rho,s}^+(x_o,0)} \eta^2 | \nabla (u-k)_\pm |^2 \, \d x \d t \\
&\leq c \iint_{Q_{\rho,s}^+(x_o,0)} (u-k)_\pm^2 |\nabla \eta|^2 \, \d x \d t
\end{align*}
for every nonnegative, time-independent piecewise smooth cutoff function $\eta$ vanishing on $\partial B_\rho(x_o)$.
\end{lemma}

\begin{proof}
The proof follows the idea from \cite[Proposition 3.2]{BDL}.
Applying Lemma \ref{l.caccioppoli} with $t_o=s$, $s-t_1$ in place of $s$ for some $t_1>0$ and $\varphi \equiv \eta$ in $B_\rho(x_o) \times (t_1,s)$, integrating over $t_1 \in (0,h)$ for some $h>0$, dividing by $h$ and using that all terms are non-negative yields that
\begin{align} \label{eq:prelim-energy-init}
	\esssup_{h < t < s} &\int_{B_\rho(x_o) \times \{t\}} \eta^2 \mathfrak{g}_\pm(u,k) \, \d x
	+ \tfrac{C_o}{2} \iint_{B_\rho(x_o) \times (h,s)} \eta^2 |\nabla (u-k)_\pm|^2 \, \d x \d t \nonumber \\
	&\leq
	c \iint_{Q_{\rho,s}^+(x_o,0)} (u-k)_\pm^2|\nabla \eta|^2 \, \d x \d t + 2 \bint_0^h \int_{B_\rho(x_o)} \eta^2 \mathfrak{g}_\pm(u,k)\, \d x \d t.
\end{align}
In the degenerate case $0<q<1$, by Lemma \ref{lem:calg} and the triangle inequality we estimate
$$
	\bint_0^h \int_{B_\rho(x_o)} \eta^2 \mathfrak{g}_\pm(u,k)
	\leq
	c(q) \|\eta\|_{L^\infty(\Omega)}^2 \bint_0^h \int_{B_\rho(x_o)} (u-k)_\pm^{q+1} \, \d x \d t.
$$
In the singular case $q>1$, by Lemma \ref{lem:calg} and H\"older's inequality we find that
\begin{align*}
	\bint_0^h \int_{B_\rho(x_o)}& \eta^2 \mathfrak{g}_\pm(u,k) \, \d x \d t\\
	&\leq
	c(q) \|\eta\|_{L^\infty(\Omega)}^2
	\bigg( \bint_0^h \int_{B_\rho(x_o)} (|u|+|k|)^{q+1} \, \d x \d t \bigg)^\frac{q-1}{q+1}\\
	&\phantom{=}\cdot
	\bigg( \bint_0^h \int_{B_\rho(x_o)} (u-k)_\pm^{q+1} \, \d x \d t \bigg)^\frac{2}{q+1}.
\end{align*}
Therefore, recalling the choice of the level $k$ and using that $u$ takes the initial condition in the sense of Definition \ref{def:localsol_pme}, we conclude that the second term on the right-hand side of~\eqref{eq:prelim-energy-init} vanishes in the limit $h \downarrow 0$ and we obtain the claimed energy estimate.
\end{proof}

In this section, we will use parameters $\muminus, \muplus \in \R$ and $\bomega > 0$ satisfying
$$
\muminus \leq \inf_{\mathcal{Q}^+} u,\quad \muplus \geq \sup_{\mathcal{Q}^+} u \quad \text{and}\quad \bomega \geq \muplus - \muminus. 
$$

\subsubsection{Tools for the case near zero}

We recall a result on propagation of positivity with pointwise information given at the initial time slice, see~\cite[Lemma 3.2]{part2}. It will be used in reduction in oscillation up to the initial boundary.

\begin{lemma} \label{l.de-giorgi-initial}
Let $u$ be a locally bounded weak sub(super)solution to \eqref{eq:pde}$_1$ in $\mathcal{Q}^+$. Set $\theta = (\xi \bomega)^{q-1}$ for some $\xi \in (0,1)$ and suppose that $Q^+_{\rho,\theta\rho^2}(x_o,t_o) \Subset \mathcal{Q}^+$. There exists a positive constant $\nu_o$ depending only on $C_o,C_1,n$ and $q$ such that if
$$
\pm (\mupm - u(x,t_o)) \geq \xi \bomega \quad \text{ a.e.~in } B_\rho(x_o),
$$
then
$$
|\mupm| \geq 8 \xi \bomega,
$$
or we have that
$$
\pm (\mupm - u) \geq \tfrac12 \xi \bomega \quad \text{ a.e. in } B_{\frac{\rho}{2}} (x_o) \times (t_o,t_o+ \nu_o \theta \rho^2).
$$
\end{lemma}

\subsubsection{Tools for the case away from zero}

Next we prove propagation in measure and a De Giorgi type lemma near the initial boundary. In the following, we suppose that 
\begin{equation} \label{eq:initial-levels}
\left\{
\begin{array}{ll}
	\muplus - \frac14 \bomega \geq \sup_{B_\rho(x_o)} g_o &\text{for weak subsolutions,}\\[5pt]
\muminus + \frac14 \bomega \leq \inf_{B_\rho(x_o)} g_o &\text{for weak supersolutions.}
\end{array}
\right.
\end{equation}
Further, we suppose that there either holds
\begin{equation} \label{eq:above-zero-initial-dg}
\muminus > \bomega \quad \text{ and } \quad \tfrac12 \muplus \leq \theta^\frac{1}{q-1} \leq 2 \muplus
\end{equation}
or
\begin{equation} \label{eq:below-zero-initial-dg}
\muplus < -\bomega \quad \text{ and } \quad \tfrac12 |\muminus| \leq \theta^\frac{1}{q-1} \leq 2 |\muminus|.
\end{equation}

First, we state an analogue to Lemma~\ref{lem:away-zero-levels}.

\begin{lemma} \label{lem:away-zero-levels-init}
Let $\eta \in \big(0,\frac18\big]$ and $k_\pm = \mupm \mp \eta \bomega$. If~\eqref{eq:above-zero-initial-dg} holds, then
$$
\tfrac12 \muplus < \muminus \leq \muplus,
$$
$$
\tfrac78 \muplus <k_+ \leq \muplus,\quad \text{and }\quad \muminus \leq k_- < \tfrac98 \muminus.
$$
Furthermore, $\theta^\frac{1}{q-1} \approx k_\pm$ up to a numerical constant.

If~\eqref{eq:below-zero-initial-dg} holds, then
$$
\muminus \leq \muplus < \tfrac12 \muminus,
$$
$$
\tfrac98 \muplus<k_+ \leq \muplus,\quad \text{and }\quad \muminus \leq k_- < \tfrac78 \muminus.
$$
Furthermore, $\theta^\frac{1}{q-1} \approx |k_\pm|$ up to a numerical constant.
\end{lemma}

\begin{lemma} \label{l.propagation-positivity-initial}
Let $u$ be a bounded weak sub(super)solution in a cylinder $Q^+_{\rho,\theta \rho^2}$ such that~\eqref{eq:pme_initial_values} and~\eqref{eq:initial-levels} are satisfied.
Then, for any $\alpha \in (0,1)$ there exists $\nu = \nu (C_o,C_1,q,n, \alpha) \in (0,1)$ such that 
$$
\left| \{ \pm(\mupm - u(\cdot,t)) \leq \tfrac18 \bomega \} \cap B_\rho(x_o) \right| \leq \alpha |B_\rho (x_o)| \quad \text{ for all } t \in (0,\nu \theta \rho^2).
$$
\end{lemma}

\begin{proof}
We only consider the case where $u$ is above zero for subsolutions, i.e.~\eqref{eq:initial-levels}$_1$ and~\eqref{eq:above-zero-initial-dg}. The remaining cases are analogous.

Let $Q = B_\rho \times (0,\nu \theta \rho^2)$ and define $k = \muplus - \tfrac14 \bomega$, $k_\eps = \muplus - \eps \bomega$. Now by choosing $\varphi$ such that $|\nabla \varphi| \leq \frac{1}{\sigma \rho}$ in the energy estimate, Lemma~\ref{lem:energy-est-initial}, for the right hand side we have
$$
\iint_{Q} (u-k)_+^2 |\nabla \varphi|^2 \, \d x \d t \leq \frac{\nu \theta \bomega^2}{(4 \sigma)^2} |B_\rho|. 
$$
For the first term on the left hand side we obtain 
\begin{align*}
\int_{B_\rho \times \{t\}} \varphi^2 \mathfrak{g}_+ (u,k) \, \d x &\geq |A_{k_\eps, (1-\sigma) \rho}(t)| \mathfrak{g}_+ (k_\eps,k) \\
&\geq c(q)|A_{k_\eps, (1-\sigma) \rho}(t)| \left( |k_\eps| + |k| \right)^{q-1} (k_\eps - k)_+^2 \\
&\geq c(q) \left( \frac{\bomega}{8}\right)^2 \theta  |A_{k_\eps, (1-\sigma) \rho}(t)|, 
\end{align*}
since $\muplus \leq |k_\eps| + |k| \leq 2 \muplus$, $\eps \in (0,\frac{1}{8}]$ and by denoting $A_{k,\rho}(t) = \{ u(\cdot,t) \geq k\}\cap B_\rho$. By using the fact that 
$$
|A_{k_\eps,\rho} (t)| \leq |A_{k_\eps, (1-\sigma)\rho} (t)| + n \sigma |B_\rho|
$$
we have
$$
|A_{k_\eps,\rho}(t)| \leq \left( c\frac{\nu}{\sigma^2} + n \sigma \right) |B_\rho|
$$
for $c = c(C_o,C_1,q) > 0$. By choosing $\sigma = \alpha (c+n)^{-1}$ and $\nu = \sigma^3$, the claim follows.
\end{proof}

\begin{lemma} \label{lem:de-giorgi-initial}
Let $Q^+_{\rho,\nu \theta \rho^2}(x_o,0)$ be a forward parabolic cylinder with $\nu \in (0,1)$, vertex in $\Omega \times \{0\}$ and $B_\rho(x_o) \Subset \Omega$. Let $u$ be a bounded weak sub(super)solution to \eqref{eq:pde}$_1$ in $\Omega_T$ such that~\eqref{eq:pme_initial_values} and~\eqref{eq:initial-levels} are satisfied. Furthermore, suppose that either~\eqref{eq:above-zero-initial-dg} or~\eqref{eq:below-zero-initial-dg} holds. Then there exists $\nu_1 = \nu_1(n,q,C_o,C_1) \in (0,1)$ such that if
$$
\left| \left\{ \pm(\mupm - u) \leq \tfrac18 \bomega \right\} \cap Q^+_{\rho,\nu \theta \rho^2} \right| \leq \nu_1 |Q^+_{\rho,\nu \theta \rho^2}|,
$$
then 
$$
\pm(\mupm -u) \geq \tfrac{1}{16} \bomega\quad \text{ a.e. in } Q^+_{\frac{\rho}{2},\nu \theta \left(\frac{\rho}{2}\right)^2}.
$$
\end{lemma}

\begin{proof}
We prove the result for subsolutions in case~\eqref{eq:above-zero-initial-dg}, i.e.~we assume that $u$ is above zero. 
The remaining cases are analogous.
We define
$$
	\begin{array}{c}
		k_j = \boldsymbol{\mu}^+ - \Big( \frac{1}{16} + \frac{1}{2^{j+4}} \Big) \boldsymbol{\omega},\,
		\rho_j = \frac{\rho}{2} + \frac{\rho}{2^{j+1}},\,
		B_j = B_{\rho_j},\,
		Q_j := Q^+_{\rho_j, \nu \theta \rho_j^2}.
	\end{array}
$$
Observe that clearly $k_j > 0$ if for all $j\in \N_0$ and~\eqref{eq:above-zero-initial-dg}$_1$ holds for all such $k_j$. Let $\varphi = \varphi(x)$ be a cutoff function such that $0\leq \varphi \leq 1$, $\varphi$ vanishes on $\partial B_j$ and takes the value $\varphi \equiv 1$ in $B_{j+1}$ and its gradient satisfies $|\nabla \varphi| \leq c \frac{2^j}{\rho}$ with a numerical constant $c$.

Since $\muplus \geq u > k_j > 0$ in the set $A_{j} = \{u > k_j\} \cap Q_j$, from the energy estimate, Lemma~\ref{lem:energy-est-initial}, we obtain that
\begin{align*}
	\min&\left\{ \left( \boldsymbol{\mu}^+ \right)^{q-1}, k_j^{q-1} \right\}
	\esssup_{0<t<\nu \theta \rho_j^2} \int_{{B}_j} \varphi^2( u-{k}_j)_+^2\, \d x 
	+ \iint_{{Q}_j} \varphi^2| \nabla (u-{k}_j)_+ |^2 \, \d x \d t \\
	&\leq
	\esssup_{0<t<\nu \theta \rho_j^2} \int_{{B}_j} \varphi^2 \left( |u| + |k_j| \right)^{q-1} ( u-{k}_j)_+^2\, \d x 
	+ \iint_{{Q}_j} \varphi^2 | \nabla (u-k_j)_+ |^2 \, \d x \d t \\
	&\leq
	c \iint_{Q_j} (u-k_j)_+^2 |\nabla \varphi|^2 \, \d x \d t \\
	&\leq
	c \frac{2^{2j}}{\rho^2} \iint_{Q_j} (u-k_j)_+^2 \, \d x \d t \\
	&\leq 
	c \frac{2^{2j} \left( \tfrac18 \boldsymbol{\omega}\right)^2}{\rho^2} \left| A_j\right|,
\end{align*}
with a constant $c = c(C_o,C_1,q)>0$. Observe that by Lemma~\ref{lem:away-zero-levels-init} we have that 
$$
\esssup_{0<t<\nu \theta \rho_j^2} \int_{{B}_j} \varphi^2( u-{k}_j)_+^2\, \d x \leq c \frac{2^{2j} \big(\tfrac18 \bomega\big)^2}{\theta \rho^2} |A_j|.
$$
Since $u - k_j \geq k_{j+1} -  k_{j} = 2^{-(j+5)} \boldsymbol{\omega}$ in the set $\{u \geq  k_{j+1}\}$, by H\"older and Sobolev inequality~\cite[Proposition 4.1, Chapter 2]{DGV} we obtain that
\begin{align*}
	\frac{\boldsymbol{\omega}}{2^{j+5}} &\left| A_{j+1} \right|
	\leq
	\iint_{{Q}_j}  (u-  k_j )_+ \varphi \, \d x \d t \\
	&\leq
	\left( \iint_{ Q_j} [ (u- k_j)_+ \varphi ]^\frac{2(n+2)}{n} \, \d x \d t \right)^\frac{n}{2(n+2)} \left| A_j \right|^{1- \frac{n}{2(n+2)}} \\
	&\leq
	c \left( \iint_{ Q_j} | \nabla [ (u- k_j)_+ \varphi ] |^2 \, \d x \d t \right)^\frac{n}{2(n+2)}
	\left( \esssup_{0 < t < \nu \theta  \rho_j^2} \int_{ B_j} \varphi^2(u- k_j)_+^2 \, \d x \right)^\frac{1}{n+2} \\
	&\phantom{=}
	\cdot \left| A_j \right|^{1- \frac{n}{2(n+2)}} \\
	&\leq
	c \left( \frac{2^{2j} \big(\tfrac18 \boldsymbol{\omega}\big)^2 }{\rho^2} \right)^\frac{n}{2(n+2)} \left( \frac{2^{2j} \big(\tfrac18 \boldsymbol{\omega}\big)^2}{\theta \rho^2} \right)^\frac{1}{n+2} \left| A_j \right|^{1+\frac{1}{n+2}} \\
	&=
	c \frac{2^{j-3} \boldsymbol{\omega}}{\theta^\frac{1}{n+2} \rho} \left| A_j \right|^{1+ \frac{1}{n+2}}
\end{align*}
for a constant $c = c(C_o,C_1,n,q)> 0$. Dividing by $|Q_{j+1}|$, denoting $Y_j = |A_j| / |Q_j|$ and using that $\nu \in (0,1)$, we have that
$$
	Y_{j+1}
	\leq
	c \nu^\frac{1}{n+2} 2^{2j} Y_{j}^{1+ \frac{1}{n+2}}
	\leq
	c 2^{2j} Y_{j}^{1+ \frac{1}{n+2}}
$$
for a constant $c = c(C_o,C_1,n,q)$. When $\nu_1 \leq c^{-(n+2)} 4^{-(n+2)^2}$,
we conclude the proof by using fast geometric convergence lemma~\cite[Lemma 5.1, Chapter 2]{DGV}.
\end{proof}

\section{Continuity up to the lateral boundary} \label{sec:continuity-lateral}
Let $(x_o,t_o)\in S_T$, $\tilde \rho_o\in (0,\min \{1,t_o,\tilde \rho\})$, where $\tilde \rho$ denotes the radius $\rho_o$ in the geometric density condition~\eqref{geometry}, and $Q := Q_{\tilde \rho_o, \tilde \rho_o}(x_o,t_o)$. Define 
$$
	\muplus_o := \sup_{Q \cap \Omega_T} u
	\quad \text{and} \quad
	\muminus_o := \inf_{Q \cap \Omega_T} u.
$$
If $0<q<1$, let
\begin{equation} \label{e.initial_omega}
\bomega_o = \max \left\{ \tfrac{2}{\xi} \|u\|_\infty, A, 4 \osc_{Q\cap \Omega_T} \psi, 4 \osc_{Q \cap S_T} g \right\},
\end{equation}
where $A > 1$ and $\xi \in (0,1)$ are constants determined by $C_o,C_1,n,q$ and $\alpha_*$ later on. 

If $q > 1$, we use the rescaling argument in Appendix~\ref{appendix-a} with $M =  \tfrac{2}{\xi} \|u\|_\infty$.  We let $\tilde \rho_o > 0$ be so small that $4 \max\{\osc_{Q \cap \Omega_T} \psi, \osc_{Q\cap S_T} g\} \leq 1$. This is possible since $\psi$ is uniformly continuous in $\overline{\Omega}_T$ and $g$ in $\overline{Q} \cap S_T$. Now we let
$$
\bomega_o = 1.
$$

In the whole range $0<q<\infty$ we let $\theta_o = \bomega_o^{q-1}$ and $\rho_o = \frac{ \tilde \rho_o}{32}$ and define cylinders
$$
Q_o := Q_{16 \rho_o, A^{(1-q)_+}\theta_o (16 \rho_o)^2}(x_o,t_o) \quad \text{ and } \quad  Q_o' :=
	Q_{\frac{\tilde \varrho_o}{2}, \frac{\tilde \rho_o}{2}}(x_o,t_o).
$$ 
It follows that 
$$
	Q_o 
	\subset
	Q'_o
	\subset 
	Q.
$$

In this section, we say that $u$ is near zero if
\begin{equation}
	\muplus_o \geq - \xi \bomega_o
	\quad \text{and} \quad
	\muminus_o \leq  \xi \bomega_o,
	\label{eq:near zero}
\end{equation}
and that $u$ is away from zero if
\begin{equation}
	\muplus_o < -  \xi \bomega_o
	\quad \text{or} \quad
	\muminus_o >   \xi \bomega_o.
	\label{eq:away from zero}
\end{equation} 

\subsection{Reduction in oscillation near zero}

Observe that with our choice of $\bomega_o$ we start in case~\eqref{eq:near zero} where $u$ is near zero. In particular, we have that $\big| \mupm_o \big| < 2 \bomega_o$. First, we suppose that
\begin{equation} \label{e.omega-alternative}
\muplus_o - \muminus_o \geq \tfrac12 \bomega_o.
\end{equation}
Together with \eqref{e.initial_omega}, we conclude that
$$
\max\left\{ \osc_{Q_o' \cap \Omega_T} \psi,\, \osc_{Q_o' \cap S_T} g \right \} = \tfrac{1}{4} \bomega_o
$$ 
or one of the following cases must hold:
\begin{equation}
	\muplus_o - \tfrac18 \bomega_o > \max \left\{ \sup_{Q_o' \cap \Omega_T} \psi,\, \sup_{Q_o' \cap S_T} g \right\}
	\quad \text{or} \quad
	\muminus_o + \tfrac18 \bomega_o < \inf_{Q_o' \cap S_T} g.
	\label{eq:comparison g}
\end{equation}
At this stage, assume that the first case of \eqref{eq:comparison g} holds and set $k := \muplus_o - \tfrac18 \bomega_o$. Now $u_k := \max\{u, k \}$ is a local weak subsolution to the obstacle free porous medium equation \eqref{eq:pde}$_1$ in $Q_o \cap \Omega_T$ by Lemma~\ref{lem:weak_subsol} (2). Note that by \eqref{eq:comparison g}$_1$ and Lemma~\ref{lem:super_sub_extension} (2),
$u_k$ can be extended from $Q_o \cap \Omega_T$ to $Q_o$ by $k$ such that
the resulting function is a weak subsolution to the porous medium equation
in $Q_o$. Thus, we can work with $Q_o$ as an interior cylinder.
In the following, we omit $(x_o,t_o)$ to simplify our notation.

Observe that the definitions of $k$ and $u_k$ and
the positive geometric density condition \eqref{geometry} imply 
\begin{align} \label{eq:measure-density-uk}
	\Big| &\Big\{ \muplus_o - u_k(\cdot,s) \geq \tfrac18 \bomega_o \Big\}
	\cap B_{ \varrho_o} \Big|  \\
	&=
	\Big| \Big\{ \muplus_o - u_k(\cdot,s) = \tfrac18 \bomega_o \Big\}
	\cap B_{ \varrho_o} \Big| 
	=
	| \{ u_k(\cdot,s) = k \} \cap B_{ \varrho_o} |
	\geq
	| B_{ \varrho_o} \setminus \Omega |
	\geq
	\alpha_* |B_{ \varrho_o}| \nonumber
\end{align}
for any $s \in \big( -A^{(1-q)_+} \theta_o  \varrho_o^2, 0 \big)$.

Let $0<q<1$, and $j_* \geq 2$ in Lemma~\ref{lem:shrinking-deg} be so large that Lemma~\ref{lem:de-giorgi-deg} is applicable, and fix $A = 2^{j_*+3}$ and $\xi = 2^{-j_*}$. Then, the aforementioned lemmas together with~\eqref{eq:measure-density-uk} yield that
$$
|\muplus_o| > \xi \bomega_o
$$
or
$$
	\muplus_o - u_k \geq \frac{1}{2^{j_*+4}} \bomega_o
	\quad \text{ a.e. in }
	Q_{\frac{{ \varrho_o}}{2}, (\xi \bomega_o)^{q-1} \left(\frac{{\varrho_o}}{2}\right)^2}
	\supset
	Q_{\frac{{ \varrho_o}}{2}, \theta_o \left(\frac{{\varrho_o}}{2}\right)^2}.
$$
If $|\muplus_o| > \xi \bomega_o$ holds, by assumption~\eqref{eq:near zero} we have that $\xi \bomega_o < \muplus_o \leq 2 \bomega_o$. In this case we use Lemma~\ref{lem:mu trapped} with $a=\tfrac12 \xi$ to conclude from~\eqref{eq:measure-density-uk} that 
$$
\muplus_o - u_k \geq \eta_1 \bomega_o \quad \text{ a.e. in } Q_{\frac{{ \varrho_o}}{2}, \theta_o \left(\frac{{\varrho_o}}{2}\right)^2}
$$
for some $\eta_1 = \eta_1(C_o,C_1,n,q,\alpha_*)> 0$.
Here, we have used that $b \in (0,1)$ in Lemma~\ref{lem:mu trapped} and $A>1$.

Let $\eta := \min\{2^{-(j_* + 4)}, \eta_1\}$ and $\delta = 1- \eta$. By similar considerations for $u - \muminus_o$ if~\eqref{eq:comparison g}$_2$ holds, we deduce
$$
\osc_{Q_{\frac{{ \varrho_o}}{2}, \theta_o \left(\frac{{\varrho_o}}{2}\right)^2} \cap \Omega_T} u \leq \delta \bomega_o.
$$
By taking the case that both conditions in~\eqref{eq:comparison g} are violated or~\eqref{e.omega-alternative} does not hold into account, we obtain that
$$
\osc_{Q_{\frac{{ \varrho_o}}{2}, \theta_o \left(\frac{{\varrho_o}}{2}\right)^2} \cap \Omega_T} u \leq \max \left\{ \delta \bomega_o, 4 \osc_{Q_o' \cap \Omega_T} \psi, 4\osc_{Q_o' \cap S_T} g \right\}.
$$

In the case $q > 1$, by taking into account Remark~\ref{rem:expansion}, we assume that $1-\epsilon \leq \frac{3}{4}$ in Proposition~\ref{prop:expansion-of-positivity-q} and let $\xi$ denote the according constant.
Thus, by means of Proposition~\ref{prop:expansion-of-positivity-q} from \eqref{eq:measure-density-uk} we deduce that
$$
|\muplus_o| > \xi \bomega_o
$$
or there exists $\eta_1$ depending only on $C_o$, $C_1$, $n$, $q$ and $\alpha_\ast$ such that
$$
	\muplus_o - u_k(\cdot,t) \geq \tfrac{1}{8}\eta_1 \bomega_o \quad \text{ a.e. in } Q_{\frac{{ \varrho_o}}{2}, \theta_o \left(\frac{{\varrho_o}}{2}\right)^2}.
$$
Again, by assumption~\eqref{eq:near zero}, $|\muplus_o| > \xi \bomega_o$ can only be satisfied if $\xi \bomega_o < \muplus_o \leq 2\bomega_o$.
In the case, we use Lemma~\ref{lem:mu trapped} with $a = \frac{1}{2}\xi$ and recall that $b \in (0,1)$ to conclude that
$$
	\muplus_o - u_k \geq \eta_2 \bomega_o \quad \text{ a.e. in }
	Q_{2\rho_o, \left(1-\tfrac12 b\right)\theta_o \varrho_o^2} \supset
	Q_{\frac{{ \varrho_o}}{2}, \theta_o \left(\frac{{\varrho_o}}{2}\right)^2}.
$$
By choosing $\eta = \min \left\{\tfrac{1}{8}\eta_1,\eta_2\right\}$ and $\delta = 1 - \eta$ we obtain that
$$
\osc_{Q_{\frac{{ \varrho_o}}{2}, \theta_o \left(\frac{{\varrho_o}}{2}\right)^2} \cap \Omega_T} u \leq \delta \bomega_o.
$$
By similar arguments for~\eqref{eq:comparison g}$_2$ and by taking the case that both conditions in~\eqref{eq:comparison g} are violated or~\eqref{e.omega-alternative} does not hold into account, we find that
$$
\osc_{Q_{\frac{{ \varrho_o}}{2}, \theta_o \left(\frac{{\varrho_o}}{2}\right)^2} \cap \Omega_T} u \leq \max \left\{ \delta \bomega_o, 4 \osc_{Q_o' \cap \Omega_T} \psi, 4\osc_{Q_o' \cap S_T} g \right\}.
$$
Now, consider the whole range $0<q<\infty$.
For the corresponding parameters $\eta$ and $\delta$ in the degenerate and singular cases, respectively, let 
$$
\bomega_1 := \max \left\{ \delta \bomega_o, 4 \osc_{Q_o' \cap \Omega_T} \psi, 4\osc_{Q_o' \cap S_T} g \right\} \quad \text{ and } \quad \theta_1 := \bomega_1^{q-1}.
$$
Furthermore, denote $Q'_1 := Q_{\frac{{ \varrho_o}}{2}, \theta_o \left(\frac{{\varrho_o}}{2}\right)^2}$, and let 
$$
Q_1 := Q_{16\rho_1, A^{(1-q)_+}\theta_1 (16\rho_1)^2}
	\quad \text{where }
	\varrho_1 := \lambda  \varrho_o
	\text{ and }
	\lambda :=
	\tfrac{1}{32} \Big( \tfrac{\delta}{A} \Big)^\frac{(1-q)_+}{2}.
$$
Since
$A^{1-q} \theta_1 ( 16\lambda  \varrho_o)^2
\leq 16^2A^{1-q} \delta^{q-1} \theta_o \lambda^2   \varrho_o^2
= \frac{1}{4} \theta_o   \varrho_o^2$
if $0<q<1$, and $\theta_1 \leq \theta_o$ if $q >1$, we find that 
$$
Q_1 \subset Q_1' \subset Q_o.
$$
Thus
$$
\osc_{Q_1 \cap \Omega_T} u \leq \bomega_1.
$$ 
Define
$$
	\muminus_1 := \inf_{Q_1 \cap \Omega_T} u
	\quad \text{and} \quad
	\muplus_1 := \muminus_1 + \bomega_1
	\geq
	\inf_{Q_1 \cap \Omega_T} u + \osc_{Q_1 \cap \Omega_T} u
	=
	\sup_{Q_1 \cap \Omega_T} u.
$$
If~\eqref{eq:near zero} holds, we may again use alternatives analogous to~\eqref{eq:comparison g} in the cylinder $Q'_1$, and proceed iteratively. In this way we can build a sequence with indices $i = 1,...,j-1$ in which~\eqref{eq:near zero} holds true up to some $j \in \N$. For each $i = 1,2,...,j$, we define
\begin{equation*}
	\begin{array}{c}
		\rho_i := \lambda \rho_{i-1},
		\quad
		\boldsymbol{\omega}_i :=
		\max \left\{ \delta \boldsymbol{\omega}_{i-1}, 4 \osc_{Q'_{i-1} \cap \Omega_T} \psi, 4 \osc_{Q'_{i-1} \cap S_T}g \right\},
		\quad
		\theta_i := \boldsymbol{\omega}_i^{q-1}, \\[5pt]
		\quad
		\lambda := \tfrac{1}{32} \left(\frac{\delta}{A}\right)^\frac{(1-q)_+}{2},
		\quad
		Q_i := Q_{16\rho_i, A^{(1-q)_+}\theta_i (16\rho_i)^2},
		\quad Q'_i = Q_{\frac{1}{2}\rho_{i-1}, \theta_{i-1} \left(\frac{1}{2}\rho_{i-1}\right)^2}, \\[5pt]
		\boldsymbol{\mu}_i^- := \inf_{Q_i} u
		\quad \text{ and } \quad
		\boldsymbol{\mu}_i^+ := \boldsymbol{\mu}_i^- + \boldsymbol{\omega}_i
	\end{array}
\end{equation*}
and deduce that for each $i = 1,2,...,j$, we have that $Q_i \subset Q_{i}'$ and that
$$
	\osc_{Q_i \cap \Omega_T} u \leq \bomega_i.
$$

\subsection{Reduction in oscillation above zero}
\label{sec:above_zero}

Suppose that $j$ is the first index for which \eqref{eq:away from zero} holds.
In this section, we assume that we are in the case \eqref{eq:away from zero}$_2$,
whereas the case \eqref{eq:away from zero}$_1$ will be treated in the next section.
Observe that $\muminus_j = \inf_{Q_j\cap \Omega_T} u$, $\muplus_j = \muminus_j + \bomega_j$,
$$
\bomega_j = \max\left\{ \delta \bomega_{j-1}, 4 \osc_{Q'_{j-1}\cap \Omega_T} \psi, 4 \osc_{Q'_{j-1} \cap S_T }g \right\}
$$ 
and $\theta_j =  \bomega_j^{q-1}$. 
Condition~\eqref{eq:away from zero}$_2$ implies that 
\begin{equation} \label{eq:mu_equiv_az}
\tfrac{\xi}{\xi+1}\muplus_j < \muminus_j \leq \muplus_j.
\end{equation}
Now, define $\theta_*=\left(\muplus_j\right)^{q-1}$.
If $q < 1$, we immediately have that $\theta_* \leq \bomega_j^{q-1}$. On the other hand, since $u$ is near zero for $j-1$, we have that
$$
 \xi \bomega_j < \muminus_j = \inf_{Q_j \cap \Omega_T}u \leq \sup_{Q_{j-1}\cap \Omega_T} u\leq \muplus_{j-1} \leq 2\bomega_{j-1} \leq \tfrac{2}{\delta} \bomega_j.	
$$
By~\eqref{eq:mu_equiv_az} we conclude that
$$
\muplus_j < \tfrac{\xi+1}{\xi} \muminus_j \leq \tfrac{2(\xi+1)}{\delta \xi} \bomega_j,
$$
such that $\theta_* < \left( \tfrac{2(\xi+1)}{\delta \xi} \right)^{q-1} \theta_j$ if $q > 1$. Denote 
\begin{equation} \label{eq:hatQj}
\hat \rho_j = \left( \tfrac{\delta \xi}{2(\xi + 1)} \right)^\frac{(q-1)_+}{2} \rho_j\quad \text{ and }\quad \widehat Q_j := Q_{\hat \rho_j, \theta_* \hat \rho_j^2}
\end{equation}
such that $\widehat Q_j \subset Q_j$.

At least one of the following must hold: $\max \{ \osc_{\widehat Q_j \cap \Omega_T} \psi, \osc_{\widehat Q_j \cap S_T} g \} = \tfrac{1}{4} \bomega_j$ or
\begin{equation}
	\muplus_j - \tfrac38 \bomega_j > \max \left\{ \sup_{\widehat Q_j \cap \Omega_T} \psi, \sup_{\widehat Q_j \cap S_T} g \right\}
	\quad \text{or} \quad
	\muminus_j + \tfrac38 \bomega_j < \inf_{\widehat Q_j \cap S_T} g.
	\label{eq:comparison g2}
\end{equation}
Suppose that~\eqref{eq:comparison g2}$_1$ holds. By setting $k := \muplus_j - \tfrac38 \bomega_j$, we have that $u_k = \max\{u,k\}$ is a weak subsolution to \eqref{eq:pde}$_1$ in $\widehat Q_j$. 
Then \eqref{geometry} implies that
$$
\left| \left\{ \muplus_j - u_k(\cdot,s) \geq  \tfrac38 \bomega_j \right\} \cap B_{\hat \rho_j} \right| \geq \alpha_* |B_{\hat \rho_j}|
$$
for all $s \in (- \theta_* \hat \rho_j^2,0)$. 
Applying Lemma~\ref{l.supersub-slicewise-to-pointwise}, we conclude that 
$$
\muplus_j- u_k \geq  a \bomega_j\quad \text{ a.e. in } Q_{\frac{\hat\rho_j}{2}, \frac12 \theta_* \left( \frac{\hat \rho_j}{2} \right)^2} \cap \Omega_T.
$$

If~\eqref{eq:comparison g2}$_2$ holds true, then we have that $u_k = \min \{u, k\}$ with $k = \muminus_j + \tfrac38 \bomega_j$ is a weak supersolution to \eqref{eq:pde}$_1$ in $\widehat Q_j$. Again, we have that
$$
\left| \left\{ u_k(\cdot,s) - \muminus_j \geq \tfrac38 \bomega_j \right\} \cap B_{\hat \rho_j} \right| \geq \alpha_* |B_{\hat \rho_j}|.
$$
In this case, using Lemma~\ref{l.supersub-slicewise-to-pointwise} shows that 
$$
u_k - \muminus_j \geq a \bomega_j \quad \text{ a.e. in } Q_{\frac{\hat \rho_j}{2}, \frac12 \theta_* \left( \frac{\hat \rho}{2} \right)^2} \cap \Omega_T.
$$
Now choose $\hat \rho_{j+1} = \hat \lambda \hat \rho_j$ with $\hat \lambda = \sqrt{\tfrac18}$. Then 
$$
	Q_{\hat \rho_{j+1}, \theta_* \hat \rho_{j+1}^2} =: Q_{j+1} \subset Q_{\frac{\hat\rho_j}{2}, \frac12 \theta_* \left( \frac{\hat\rho_j}{2} \right)^2} \subset Q_j 
$$
and
$$
\osc_{Q_{j+1}\cap \Omega_T} u \leq \delta \bomega_j.
$$
On the other hand, if \eqref{eq:comparison g2} fails, we have that
$$
\osc_{Q_{j+1} \cap \Omega_T} u \leq \bomega_j = 4 \max\left\{ \osc_{\widehat Q_j \cap \Omega_T} \psi, \osc_{\widehat Q_j \cap S_T}g \right\},
$$
so that altogether
$$
\osc_{Q_{j+1} \cap \Omega_T} u  \leq \max \left\{ \delta \bomega_j , 4 \osc_{Q_j' \cap \Omega_T} \psi, 4 \osc_{Q_j' \cap S_T}g \right\} =: \bomega_{j+1}.
$$
Again, define $\muminus_{j+1} = \inf_{Q_{j+1} \cap \Omega_T} u$ and $\muplus_{j+1} = \muminus_{j+1} + \bomega_{j+1}$.
Observe that since $\bomega_{j+1} \leq \bomega_j$ we have that
$$
 \xi \bomega_{j+1} \leq  \xi \bomega_j < \muminus_j \leq \muminus_{j+1},
$$
which implies that we are again in case~\eqref{eq:away from zero}$_2$ for $j+1$, and also that
$$
\tfrac12 \muplus_{j+1} \leq \theta_*^\frac{1}{q-1} \leq \tfrac{\xi+1}{\xi} \muplus_{j+1}.
$$
Now, we can iterate with the choices
$$
Q_i:= Q_{\hat \rho_i,\theta_* \hat \rho_i^2} \quad \text{ with } \hat \rho_i = \hat\lambda^{i-j} \left( \tfrac{\delta \xi}{2(\xi + 1)} \right)^\frac{(q-1)_+}{2} \rho_{j}
\quad\text{and}\quad \hat\lambda := \sqrt{\tfrac{1}{8}},
$$
which yields that for any $i > j$, there holds
$$
	\osc_{Q_i \cap \Omega_T} u
	\leq
	\bomega_i.
$$

\subsection{Reduction in oscillation below zero}

Suppose that \eqref{eq:away from zero}$_1$ holds, which implies that
\begin{equation} \label{eq:mu_equiv_bz}
\muminus_j \leq \muplus_j < \tfrac{\xi}{\xi + 1} \muminus_j.
\end{equation}
Recall $\muminus_j = \inf_{Q_j\cap \Omega_T} u$, $\bomega_j = \max\{\delta \bomega_{j-1}, 4 \osc_{Q_{j-1}\cap \Omega_T}\psi, 4 \osc_{Q_{j-1}\cap S_T} g\}$, $\muplus_j = \muminus_j + \bomega_j $ and $\theta = \bomega_j^{q-1}$. Consider $\theta_* = |\muminus_j|^{q-1} $ analogously as in the case where $u$ is above zero. If $q <1$ we have $\theta_* \leq \theta$. If $q> 1$, it follows that 
$$
- \muminus_j < - \tfrac{\xi + 1}{\xi} \muplus_j \leq \tfrac{2(\xi + 1)}{\delta \xi} \bomega_j,
$$
such that $\theta_* < \left(\tfrac{2(\xi + 1)}{\delta \xi}\right)^{q-1} \theta$. Define $\hat \rho_j$ and $\widehat Q_j$ analogous to~\eqref{eq:hatQj}.
Again, consider the alternatives that $\max \{ \osc_{\widehat Q_j \cap \Omega_T} \psi, \osc_{\widehat Q_j \cap S_T} g \} = \tfrac{1}{4} \bomega_j$ or that~\eqref{eq:comparison g2} holds.
As in Section~\ref{sec:above_zero}, we define $Q_{j+1}$ and by means of Lemma~\ref{l.supersub-slicewise-to-pointwise} we infer that
$$
\osc_{Q_{j+1} \cap \Omega_T} u \leq \max\left\{ \delta \bomega_j, 4 \osc_{\widehat Q_j \cap \Omega_T} \psi,4 \osc_{\widehat Q_j \cap S_T} g \right\} =: \bomega_{j+1}.
$$
Next, we define $\muplus_{j+1} = \sup_{Q_{j+1} \cap \Omega_T} u$, $\muminus_{j+1} = \muplus_{j+1}- \bomega_{j+1}$. 
Now we can deduce that 
$$
\muplus_{j+1} \leq \muplus_j < - \xi \bomega_j \leq - \xi \bomega_{j+1}.
$$
Further, we have that
$$
\tfrac12 |\muminus_{j+1}| \leq \theta_*^\frac{1}{q-1} \leq \tfrac{\xi + 1}{\xi}| \muminus_{j+1}|.
$$
Consider the alternatives 
\begin{equation*}
	\muplus_{j+1} - \tfrac38 \bomega_{j+1} > \max \left\{ \sup_{Q_{j+1} \cap \Omega_T} \psi, \sup_{Q_{j+1} \cap S_T}g \right\}
	\quad \text{or} \quad
	\muminus_{j+1} + \tfrac38 \bomega_{j+1} < \inf_{Q_{j+1} \cap S_T} g.
\end{equation*}
At this stage, for $i > j$ we iterate the arguments in this section with the choices
$$
Q_i:= Q_{\hat \rho_i,\theta_* \hat \rho_i^2} \quad \text{ with } \hat \rho_i = \hat\lambda^{i-j} \left( \tfrac{\delta \xi}{2(\xi + 1)} \right)^\frac{(q-1)_+}{2} \rho_{j},\ \hat\lambda := \sqrt{\tfrac{1}{8}}.
$$
This shows that for any $i>j$, we have that
$$
	\osc_{Q_i \cap \Omega_T} u \leq \bomega_i.
$$

\subsection{Oscillation decay estimate up to the lateral boundary} \label{subsec:oscillation-decay-lateral}

We define
$$
	r_i :=
	\left\{
	\begin{array}{ll}
		\left( \tfrac{\delta \xi^2}{2(\xi+1)} \right)^\frac{q-1}{2} \lambda_s^i \rho_o
		&\text{if } q > 1, \\
		\left( \tfrac{2(\xi+1)}{\xi} \right)^\frac{q-1}{2}  \lambda^i \rho_o
		&\text{if } 0<q<1,
	\end{array}
	\right.
$$
where $\lambda_s = \delta^\frac{q-1}{2}\lambda$.
Observe that $r_i \leq \hat \rho_i$ for every $i \in \N_0$.
We claim that
$$
Q_{r_i} := Q_{r_i, \theta_o  r_i^2}\subset Q_i
$$
for any $i \in \N_0$. If $0<q<1$ and $i \leq j$, we have that $\theta_i = \bomega_{i}^{q-1} \geq \bomega_o^{q-1} = \theta_o$. If $0<q<1$ and $i > j$, there holds $\theta_* \geq \left( \frac{2(\xi + 1)}{\xi} \right)^{q-1} \bomega_o^{q-1} = \left( \frac{2(\xi + 1)}{\xi} \right)^{q-1}\theta_o$. Next, if $q>1$ and $i \leq j$, we have that $\theta_i = \bomega_i^{q-1} \geq \left(\delta^i \bomega_o\right)^{q-1}$. If $q>1$ and $i > j$, there holds $\theta_* \geq \left( \xi \bomega_j \right)^{q-1} \geq \left( \xi \delta^j \bomega_o \right)^{q-1}$.

Collecting the results from the three preceding sections, we infer that
\begin{align*}
\osc_{Q_{r_i, \theta_o  r_i^2}} u \leq \bomega_i &\leq \delta^i \bomega_o + 4 \sum_{k=0}^{i-1} \delta^k \left( \osc_{Q_{i-k-1}^{(\prime)}\cap \Omega_T}\psi + \osc_{Q_{i-k-1}^{(\prime)}\cap S_T} g \right) \\
&\leq \delta^i \bomega_o + \frac{4}{1-\delta} \left[\bomega_{\psi}(\tilde \rho_o ) + \bomega_g (\tilde \rho_o) \right],
\end{align*}
where $Q_{i-k-1}^{(\prime)}$ denotes $Q_{i-k-1}'$ for $i-k-1 \leq j$ and $Q_{i-k-1}$ for $i-k-1>j$.
Suppose that $0<q<1$. Recalling the definition of $r_i$, we find that
$$
\lambda^i = \left( \tfrac{2(\xi+1)}{\xi} \right)^\frac{1-q}{2} \frac{r_i}{\rho_o}.
$$
If $q > 1$ we have in a similar fashion that
$$
\lambda_s^i = \left( \tfrac{\delta \xi^2}{2(\xi+1)} \right)^\frac{1-q}{2} \frac{r_i}{\rho_o}.
$$
In the following, by the short notation $\lambda_{(s)}$ we mean $\lambda$ if $0<q<1$, and $\lambda_s$ if $q>1$.
Hence, for $\gamma_1 = \frac{\log \delta}{\log \lambda_{(s)}}$ we find that
\begin{align*}
\osc_{Q_{r_i, \theta_o  r_i^2}} u &\leq \lambda_{(s)}^{\gamma_1 i }\bomega_o + \frac{4}{1-\delta} \left[\bomega_{\psi}(\tilde \rho_o) + \bomega_g (\tilde \rho_o) \right] \\
&\leq c \left(\frac{r_i}{\rho_o} \right)^{\gamma_1} \bomega_o + \frac{4}{1-\delta} \left[\bomega_{\psi}(\tilde \rho_o) + \bomega_g (\tilde \rho_o) \right],
\end{align*}
where $c= c(C_o,C_1,n,q,\alpha_*)$. Let $r \in (0,r_o)$. Then there exists $i \in \N_0$ such that $r_{i+1} < r \leq r_i$. Since $r_i/r_{i+1} =  \lambda_{(s)}^{-1}$, it immediately follows that 
$$
\osc_{Q_{r, \theta_o  r^2}} u \leq c \left(\frac{r}{\rho_o} \right)^{\gamma_1} \bomega_o + \frac{4}{1-\delta} \left[\bomega_{\psi}(\tilde \rho_o) + \bomega_g (\tilde \rho_o) \right],
$$
where $c= c(C_o,C_1,n,q,\alpha_*)$. In a similar fashion, for $r \in [r_o,\rho_o)$ we obtain that
$$
\osc_{Q_{r, \theta_o  r^2}} u \leq \bomega_o  \leq \left( \tfrac{2(\xi+1)}{\xi} \right)^\frac{\gamma_1(1-q)}{2} \bomega_o \left(\frac{r}{\rho_o} \right)^{\gamma_1}
$$
if $0<q<1$ and
$$
\osc_{Q_{r, \theta_o  r^2}} u \leq \bomega_o  \leq \left( \tfrac{\delta \xi^2}{2(\xi+1)} \right)^\frac{\gamma_1(1-q)}{2} \bomega_o \left(\frac{r}{\rho_o} \right)^{\gamma_1}
$$
if $q>1$. Thus, altogether we have  that
$$
\osc_{Q_{r, \theta_o  r^2}} u \leq c \left(\frac{r}{\rho_o} \right)^{\gamma_1} \bomega_o + \frac{4}{1-\delta} \left[\bomega_{\psi}(\tilde \rho_o) + \bomega_g (\tilde \rho_o) \right]
$$
for every $r \in (0,\rho_o)$. Without loss of generality, we can replace $\rho_o$ by any $\tilde \rho \in (r,\rho_o)$. By choosing $\tilde \rho = \sqrt{r\rho_o}$ and by observing that $\tilde \rho_o = 32 \rho_o$, for any $r \in (0,\varrho_o)$ we obtain the desired oscillation decay estimate
$$
\osc_{Q_{r, \theta_o  r^2}} u \leq c \left(\frac{r}{\rho_o} \right)^\frac{\gamma_1} {2}\bomega_o + \frac{4}{1-\delta} \left[\bomega_{\psi}(32\sqrt{r \rho_o}) + \bomega_g (32 \sqrt{r \rho_o}) \right].
$$
This concludes the proof of the claim for $(x_o,t_o) \in S_T$ in Theorem~\ref{thm:main_theorem}.

\begin{remark} \label{rem:lateral-distance}
Observe that $\tilde \rho_o$ depends on the instance $t_o$ and the estimate in Theorem~\ref{thm:main_theorem} will depend on the distance to the boundary $\Omega \times \{0\}$ in this case. This is due to the fact that lateral and initial boundary data may not be compatible (i.e., continuous at the corner points). If we suppose $g \in C(\Omega_T \cup \partial_p \Omega_T)$ and $g_o \equiv g(\cdot,0)$, we do not need this restriction. Indeed, by replacing oscillation and extrema of $g$ over $Q \cap S_T$ by $Q \cap \partial_p \Omega_T$ in the argument, we can extend $u_k$ by $k$ to the negative times in the proof, which will be again a weak sub(super)solution to the obstacle free problem in the whole cylinder $Q$.
\end{remark}

\section{Continuity up to the initial boundary} \label{sec:initial}

Consider a cylinder $Q = Q_{\rho_o,\rho_o}^+(x_o,0) \subset \Omega_T$ with vertex $(x_o,0) \in \Omega \times \{0\}$ and $\rho_o < 1$. We set $x_o = 0$ and define
$$
\muplus_o = \sup_{Q} u, \quad \muminus_o = \inf_{Q} u.
$$
If $0<q<1$, let 
\begin{equation}
\bomega_o =\max \left\{2\|u\|_\infty,4,4 \osc_{Q} \psi, 4\osc_{B_{\rho_o}} g_o \right\}.
\end{equation}
If $q>1$, we use the rescaling argument in Appendix~\ref{appendix-a} with $M = 2 \|u\|_\infty$ and let
$$
\bomega_o := 1.
$$ 
Again, by uniform continuity of $\psi$ and $g_o$ we suppose that $\rho_o$ is so small that $4 \osc_{Q} \psi \leq 1$ and $4\osc_{B_{\rho_o}} g_o \leq 1$ independently of the point $x_o \in \Omega$.
Let $\theta_o = \left( \tfrac14 \bomega_o \right)^{q-1}$.
Then, in the whole range $0<q<\infty$ we have that
$$
Q_o := Q_{\rho_o, \theta_o \rho_o^2}^+(x_o,0) \subset Q \quad \text{ and }\quad \osc_{Q_o} u \leq \bomega_o.
$$

In the following, we will omit $(x_o,0)$ in order to simplify our notation.
Further, we say that $u$ is near zero if
\begin{equation} \label{e.nearzero-initial}
\muminus_o \leq \bomega_o \quad \text{ and } \quad \muplus_o \geq -\bomega_o
\end{equation} 
and away from zero if
\begin{equation} \label{e.awayzero-initial}
\muminus_o > \bomega_o \quad \text{ or } \quad \muplus_o < -\bomega_o.
\end{equation} 

\subsection{Reduction in oscillation near zero}
Observe that with our choice of $\bomega_o$ we start in the case where $u$ is near zero. This implies $|\mupm_o| \leq 2\bomega_o$. First, we suppose that
\begin{equation} \label{e.omega-alternative-init}
\muplus_o - \muminus_o \geq \tfrac34 \bomega_o.
\end{equation}
In the following, we use the alternatives
\begin{equation}
	\left\{
	\begin{array}{c}
	\muplus_o - \tfrac14 \bomega_o
	> \max \left\{ \sup_{Q_o} \psi, \sup_{B_{\rho_o}} g_o \right\}
	\quad \text{or} \quad
	\muminus_o + \tfrac14 \bomega_o < \inf_{B_{\rho_o}} g_o \\[5pt]
	\quad \text{or} \quad
	\bomega_o = 4 \max \left\{ \osc_{Q_o} \psi, \sup_{B_{\rho_o}} g_o \right\}
	\end{array}
	\right.
	\label{eq:alternatives-initial}
\end{equation}

If~\eqref{eq:alternatives-initial}$_1$ holds, let $k = \muplus_o - \tfrac14 \bomega_o$.
Then, by Lemma~\ref{lem:weak_subsol}~\eqref{lem:weak_subsol_2} $u_k := \max \{u,k\}$ is a weak subsolution to \eqref{eq:pde}$_1$.
Choosing $t_o = 0$ and $\xi = \frac14$ in Lemma~\ref{l.de-giorgi-initial}, we deduce that
$$
u \leq u_k \leq \muplus_o - \tfrac18 \bomega_o \quad \text{ a.e in }
Q_{\frac{\rho_o}{2}, \nu_o \theta_o \rho_o^2}^+.
$$
If~\eqref{eq:alternatives-initial}$_2$ holds, we set $k= \muminus_o + \tfrac14 \bomega_o$ and proceed in a similar way.
By denoting $\widetilde Q := Q_{\frac{\rho_o}{2}, \nu_o \theta_o \rho_o^2}^+ \subset Q_o$, taking into account that~\eqref{eq:alternatives-initial}$_3$ may hold or~\eqref{e.omega-alternative-init} is violated and combining the results, we have that
$$
	\osc_{\widetilde Q} u \leq \max \left\{ \tfrac78 \bomega_o, 4 \osc_{Q} \psi, 4\sup_{B_{\rho_o}} g_o \right\} =: \bomega_1.
$$
We define $\theta_1 := \big(\tfrac14 \bomega_1\big)^{q-1}$ and $\rho_1 := \lambda \rho_o$, where $\lambda := \tfrac12 \left( \frac78 \right)^\frac{(1-q)_+}{2}\nu_o^\frac{1}{2}$. Now, we set
$$
Q_1 := Q_{\rho_1, \theta_1 \rho_1^2}^+ \subset \widetilde Q.
$$

At this stage, we iterate the preceding arguments.
More precisely, for all indices $i = 1,...,j-1$ for which~\eqref{e.nearzero-initial} is satisfied we define
\begin{equation*}
	\left\{
	\begin{array}{c}
		 \rho_i :=  \lambda  \rho_{i-1},
		\quad
		\boldsymbol{\omega}_i :=
		\max \left\{ \tfrac{7}{8} \boldsymbol{\omega}_{i-1}, 4 \osc_{Q_{i-1}} \psi, 4 \osc_{B_{\rho_{i-1}}}g_o \right\},\\[5pt]
		\theta_i := \big(\tfrac14 \bomega_i\big)^{q-1},
		\quad
		 \lambda := \tfrac12 \left( \frac78 \right)^\frac{(1-q)_+}{2}\nu_o^\frac{1}{2},
		\quad
		Q_i := Q_{\rho_i, \theta_i \rho_i^2}^+,\\[5pt]
		\boldsymbol{\mu}_i^- := \inf_{Q_i} u
		\quad \text{ and } \quad
		\boldsymbol{\mu}_i^+ := \boldsymbol{\mu}_i^- + \boldsymbol{\omega}_i.
	\end{array}
	\right.
\end{equation*}
Then, for all $i = 1,...,j$ we obtain that
$$
	\osc_{Q_i} u \leq \bomega_i.
$$

\subsection{Reduction in oscillation above zero}
\label{sec:above_zero_initial}

Consider the first index $j$ for which \eqref{e.nearzero-initial} is false.
In this section, we assume that~\eqref{e.awayzero-initial}$_1$ holds, i.e.~$u$ is above zero.
For~\eqref{e.awayzero-initial}$_2$ we refer to the next section.
Since~\eqref{e.nearzero-initial} holds true for $j-1$, for $\delta = \frac78$ we have that 
$$
\tfrac12 \muplus_j <\muminus_j \leq \muplus_j,
$$
and
$$
\bomega_j < \muminus_j \leq \muplus_{j-1} \leq 2 \bomega_{j-1} \leq \tfrac{2}{\delta} \bomega_j,
$$
which imply that
$$
\muplus_j < \tfrac{4}{\delta} \bomega_j.
$$
Let $\theta_* = (\muplus_j)^{q-1}$. If $0<q<1$, then $\theta_* \leq \theta_j$, and if $q  >1$, then $\theta_* < \left(\tfrac{16}{\delta} \right)^{q-1} \theta_j$.

We use the alternatives 
\begin{equation}
	\left\{
	\begin{array}{c}
	\muplus_j - \tfrac14 \bomega_j
	> \max \left\{ \sup_{Q_j} \psi, \sup_{B_{\rho_j}} g_o \right\}
	\quad \text{or} \quad
	\muminus_j + \tfrac14 \bomega_j < \inf_{B_{\rho_j}} g_o \\[5pt]
	\quad \text{or} \quad
	\bomega_j \leq 4 \max \left\{ \osc_{Q_j} \psi, \sup_{B_{\rho_j}} g_o \right\}.
	\end{array}
	\right.
	\label{eq:alternatives-initial-j}
\end{equation}
Suppose  that alternative~\eqref{eq:alternatives-initial-j}$_1$ holds. Again $u_k = \max \{u,k\}$ for $k = \muplus_j - \tfrac14 \bomega_j$ is a weak subsolution in $Q_j$. Define $\hat \rho_j = \left( \tfrac{\delta}{16} \right)^\frac{(q-1)_+}{2} \rho_j$.
From Lemma~\ref{l.propagation-positivity-initial} we obtain that for any $\alpha \in(0,1)$ there exists $\nu = \nu (C_o,C_1,n,q,\alpha)\in (0,1)$ such that 
$$
\left| \big\{ \muplus_j - u_k \leq \tfrac18 \bomega_j \big\} \cap Q_{\hat \rho_j, \nu \theta_* \hat \rho_j^2}^+ \right| \leq \alpha |Q_{\hat \rho_j, \nu \theta_* \hat \rho_j^2}^+|.
$$
At this point, let $\alpha = \nu_1(n,q,C_o,C_1) \in (0,1)$ in Lemma~\ref{lem:de-giorgi-initial} (which also fixes $\nu = \nu(n,q,C_o,C_1) \in (0,1)$ in the length of the cylinders in the preceding inequality) such that Lemma~\ref{lem:de-giorgi-initial} is applicable. Then we can conclude 
$$
 u \leq u_k \leq \muplus_j - \tfrac{1}{16} \bomega_j\quad \text{ a.e. in } Q_{\frac{\hat \rho_j}{2}, \nu \theta_* \left( \frac{\hat \rho_j}{2} \right)^2}^+.
$$
If alternative~\eqref{eq:alternatives-initial}$_2$ holds, by similar arguments for $u_k = \min \{u,k\}$ with $k = \muminus_j + \frac14 \bomega_j$ we obtain that
$$
u \geq u_k \geq \muminus_j + \tfrac{1}{16}\bomega_j\quad \text{ a.e. in } Q_{\frac{\hat \rho_j}{2}, \nu \theta_* \left( \frac{\hat \rho_j}{2} \right)^2}^+.
$$

Define $Q_{j+1} := Q_{\hat \rho_{j+1}, \theta_* \hat \rho_{j+1}^2}^+$, where $\hat \rho_{j+1} = \hat \lambda \hat \rho_j$ with $\hat \lambda := \tfrac12 \nu^\frac12$. 
By taking into account that~\eqref{eq:alternatives-initial}$_3$ may hold, we obtain
$$
\osc_{Q_{j+1}} u \leq \max\Big\{ \tfrac{15}{16} \bomega_j, 4 \osc_{Q_j} \psi, 4 \osc_{B_{\rho_j}} g_o \Big\} =: \bomega_{j+1}.
$$
Define $\muminus_{j+1} = \inf_{Q_{j+1}} u$ and $\muplus_{j+1} = \muminus_{j+1} + \bomega_{j+1}$. Observe that 
$$
\muminus_{j+1} \geq \muminus_{j} > \bomega_j \geq \bomega_{j+1},
$$
which implies that for index $j+1$ \eqref{e.awayzero-initial}$_1$ holds. Furthermore, we have 
$$
\tfrac12 \muplus_{j+1} \leq \theta_*^\frac{1}{q-1} \leq 2 \muplus_{j+1}.
$$
Let $B_{j} = B_{\rho_j}$. Iterating the arguments above, for any $i > j$ we define
\begin{equation*}
	\left\{
	\begin{array}{c}
		\hat \rho_i := \hat \lambda \hat \rho_{i-1},
		\quad
		\hat \lambda := \tfrac12 \nu^\frac{1}{2},
		\quad
		B_{i} = B_{\hat \rho_i},
		\quad
		Q_i := Q_{\hat \rho_i, \theta_* \hat\rho_i^2}^+,
 \\[5pt]
		\boldsymbol{\omega}_i :=
		\max\left\{ \tfrac{15}{16} \boldsymbol{\omega}_{i-1}, 4 \osc_{Q_{i-1}} \psi, 4 \osc_{B_{i-1}}g_o \right\},
		\\[5pt]
		\boldsymbol{\mu}_i^- := \inf_{Q_i} u
		\quad \text{ and } \quad
		\boldsymbol{\mu}_i^+ := \boldsymbol{\mu}_i^- + \boldsymbol{\omega}_i
	\end{array}
	\right.
\end{equation*}
and obtain for any $i > j$ that
$$
	\osc_{Q_i} u \leq \bomega_i.
$$

\subsection{Reduction in oscillation below zero}

Suppose that $j\in \N$ is the first index such that~\eqref{e.awayzero-initial}$_2$ holds true. This implies that
$$
\muminus_j \leq \muplus_j < \tfrac12 \muminus_j.
$$
Since~\eqref{e.nearzero-initial} is fulfilled for $j-1$, we have that
$$
-\bomega_j >\muplus_j \geq \sup_{Q_j} u \geq \inf_{Q_{j-1}} u = \muminus_{j-1} \geq -2 \bomega_{j-1} \geq -\tfrac{2}{\delta} \bomega_j,
$$
where $\delta = \frac78$. By combining the estimates above we obtain that
$$
- \muminus_j \leq \tfrac{4}{\delta} \bomega_j.
$$
Define $\theta_* = |\muminus_j|^{q-1}$. If $0<q<1$, then $\theta_* \leq \theta_j$, and if $q>1$, then $\theta_* < \left( \tfrac{16}{\delta} \right)^{q-1} \theta_j$.
Using Lemmas~\ref{l.propagation-positivity-initial} and~\ref{lem:de-giorgi-initial} for alternatives~\eqref{eq:alternatives-initial-j}$_{1,2}$ as in Section~\ref{sec:above_zero_initial}, defining $Q_{j+1} := Q_{\hat \rho_{j+1}, \theta_* \hat \rho_{j+1}^2}^+$, where $\hat \rho_{j+1} = \hat \lambda \hat \rho_j$ with $\hat \lambda := \tfrac12 \nu^\frac12$ and $\hat \rho_j = \left( \tfrac{\delta}{16} \right)^\frac{(q-1)_+}{2} \rho_j$, and and taking \eqref{eq:alternatives-initial-j}$_3$ into account, we conclude that
$$
\osc_{Q_{j+1}} u \leq \max\Big\{ \tfrac{15}{16} \bomega_j, 4 \osc_{Q_j} \psi, 4 \osc_{B_{\rho_j}} g_o \Big\} =: \bomega_{j+1}.
$$
Define $\muplus_{j+1} = \sup_{Q_{j+1}} u$ and $\muminus_{j+1} = \muplus_{j+1} - \bomega_{j+1}$. Observe that 
$$
\muplus_{j+1} \leq \muplus_j < - \bomega_j \leq - \bomega_{j+1},
$$
which implies that~\eqref{e.awayzero-initial}$_2$ holds for index $j+1$. Furthermore
$$
\tfrac12 |\muminus_{j+1}| \leq \theta_*^\frac{1}{q-1} \leq 2 |\muminus_{j+1}|.
$$
Now we can repeat the arguments above. Let $B_{j} = B_{\rho_j}$ and for any $i > j$ define
\begin{equation*}
	\left\{
	\begin{array}{c}
		\hat \rho_i := \hat \lambda \hat \rho_{i-1},
		\quad
		\hat \lambda := \tfrac12 \nu^\frac{1}{2},
		\quad
		B_{i} = B_{\hat \rho_i},
		\quad
		Q_i := Q_{\hat \rho_i, \theta_* \hat\rho_i^2}^+,
 \\[5pt]
		\boldsymbol{\omega}_i :=
		\max\left\{ \tfrac{15}{16} \boldsymbol{\omega}_{i-1}, 4 \osc_{Q_{i-1}} \psi, 4 \osc_{B_{i-1}}g_o \right\},\\[5pt]
		\boldsymbol{\mu}_i^+ := \sup_{Q_i} u
		\quad \text{ and } \quad
		\boldsymbol{\mu}_i^- := \boldsymbol{\mu}_i^+ - \boldsymbol{\omega}_i
	\end{array}
	\right.
\end{equation*}
and conclude that for any $i > j$ there holds
$$
	\osc_{Q_i} u \leq \bomega_i.
$$

\subsection{Oscillation decay estimate up to the initial boundary}

We define
$$
	r_i :=
	\left\{
	\begin{array}{ll}
		\left( \tfrac{\delta}{16} \right)^\frac{q-1}{2} \lambda_s^i \rho_o
		&\text{if } q > 1, \\
		\left( \tfrac{16}{\delta} \right)^\frac{q-1}{2}  \lambda^i \rho_o
		&\text{if } 0<q<1,
	\end{array}
	\right.
$$
where $\lambda_s = \delta^\frac{q-1}{2} \lambda$, $\lambda = \tfrac12 \delta^\frac{(1-q)_+}{2} \min\{\nu_o,\nu\}^\frac{1}{2}$ and $\delta = \frac78$.
Observe that $r_i \leq \hat \rho_i$ for every $i \in \N_0$.
We claim that
$$
Q_{r_i, \theta_o r_i^2}^+ \subset Q_i
$$
for any $i \in \N_0$. If $0<q<1$ and $i \leq j$, we have $\theta_i = \left( \frac14 \bomega_{i}\right)^{q-1} \geq \left( \frac14 \bomega_o\right)^{q-1} = \theta_o$. If $0<q<1$ and $i > j$, there holds $\theta_* \geq \left( \frac{16}{\delta} \frac14 \bomega_o \right)^{q-1}  = \left( \frac{16}{\delta} \right)^{q-1}\theta_o$. If $q>1$ and $i \leq j$, we have $\theta_i = ( \frac14 \bomega_i)^{q-1} \geq \left(\delta^i \frac14 \bomega_o\right)^{q-1}$. If $q>1$ and $i > j$, there holds $\theta_* \geq \left( \frac14 \bomega_j\right)^{q-1} \geq \left( \delta^j \frac14 \bomega_o \right)^{q-1}$.

By denoting $\tilde \delta = \frac{15}{16}$, we have that
\begin{align*}
\osc_{Q_{r_i, \theta_o r_i^2}^+} u \leq \bomega_i &\leq \tilde \delta^i \bomega_o + 4 \sum_{k=0}^{i-1} \tilde\delta^k \left( \osc_{Q_{i-k-1}}\psi + \osc_{B_{i-k-1}} g_o \right) \\
&\leq \tilde \delta^i \bomega_o + \frac{4}{1-\tilde \delta} \left[\bomega_{\psi}(\rho_o) + \bomega_{g_o} (\rho_o) \right].
\end{align*}
Hence, arguing as in Section~\ref{subsec:oscillation-decay-lateral} for any $r \in (0,\varrho_o)$ we arrive at
$$
\osc_{Q_{r, \theta_o r^2}^+} u \leq c \left(\frac{r}{\rho_o} \right)^\frac{\gamma_1} {2}\bomega_o + \frac{4}{1-\tilde\delta} \left[\bomega_{\psi}(\sqrt{r \rho_o}) + \bomega_{g_o} (\sqrt{r \rho_o}) \right].
$$
This proves the claim of Theorem~\ref{thm:main_theorem} regarding vertices $(x_o,0) \in \Omega \times \{0\}$.

\begin{remark}
In an analogous way as in Section~\ref{sec:continuity-lateral}, the final estimate in Theorem~\ref{thm:main_theorem} depends on the distance to the boundary $S_T$. If the boundary data are compatible, this can be avoided with similar arguments as in Remark~\ref{rem:lateral-distance}.
\end{remark}

\appendix

\section{Rescaling argument} \label{appendix-a}

Let $M>0$ and consider the functions
\begin{align} \label{e.rescaled_functions}
	\tilde{u}(x,t) :=
	\frac{1}{M} u(x, &M^{q-1} t), 
	\quad	
	\widetilde{\psi}(x,t) :=
	\frac{1}{M} \psi(x, M^{q-1} t),
	\, \text{ and } \\
	&\tilde{g}(x,t) :=
	\frac{1}{M} g(x, M^{q-1} t) \nonumber
\end{align}
together with
\begin{align} \label{e.rescaled_A}
	\widetilde{\mathbf{A}}(x,t,\boldsymbol{v}^q,\zeta)
	:=
	\frac{1}{M} \mathbf{A}(x, M^{q-1}t, (\boldsymbol{Mv})^q, M\zeta)
\end{align}
for $(x,t) \in \Omega_{\widetilde{T}} := \Omega \times (0,\widetilde{T}) := \Omega \times (0, M^{1-q} T)$, such that the vector field $\widetilde{\mathbf{A}}$ satisfies the same structure conditions as $\mathbf{A}$.

\begin{lemma} \label{l.rescaling}
Let $\tilde{u}$, $\widetilde{\psi}$ and $\tilde{g}$ be defined as in~\eqref{e.rescaled_functions} and $\widetilde{\mathbf{A}}$ as in~\eqref{e.rescaled_A}. Then $\tilde{u}$ is a weak solution to the obstacle problem
with obstacle $\widetilde{\psi}$ and boundary values $\tilde{g}$ such that
$$
	\partial_t \boldsymbol{\tilde{u}}^q
	- \Div \widetilde{\mathbf{A}}(x,t,\boldsymbol{\tilde{u}}^q, \nabla\tilde{u}) = 0
	\quad \text{in } \Omega_{\widetilde{T}}
$$
in the sense of Definition~\ref{d.obstacle-wsol}.
\end{lemma}

\begin{proof}
The variational inequality in Definition~\ref{d.obstacle-wsol} is satisfied with similar arguments as in~\cite{MS}. By denoting $\tilde{t} = M^{1-q} t$, for $g$ with modulus of continuity $\bomega_g$ on $S_T$ we have that 
\begin{align*}
|\tilde{g}(x,\tilde t) - \tilde{g}(y,\tilde s)| \leq \frac{1}{M} \bomega_g\left(|x-y| + |M^{q-1} \tilde t- M^{q-1} \tilde s|^\frac{1}{2}\right)
\end{align*}
for every $(x,\tilde t), (y, \tilde s) \in S_{\widetilde T}$. Similarly, we find that
\begin{align*}
|\widetilde{\psi}(x,\tilde t) - \widetilde{\psi}(y,\tilde s)| \leq \frac{1}{M} \bomega_\psi \left(|x-y| + |M^{q-1} \tilde t- M^{q-1} \tilde s|^\frac{1}{2}\right)
\end{align*}
for every $(x,\tilde t), (y, \tilde s) \in \Omega_{\widetilde T}$ and 
\begin{align*}
|\tilde{g}_o(x) - \tilde{g}_o(y)| \leq \frac{1}{M} \bomega_{g_o}\left(|x-y| \right)
\end{align*}
for every $x,y \in \Omega$. Also, $\tilde u (\cdot,\tilde t) - \tilde g (\cdot,\tilde t) = \frac{1}{M} \left( u(\cdot,M^{q-1}\tilde t) - g(\cdot, M^{q-1}\tilde t) \right) \in H^1_0(\Omega)$ for a.e. $\tilde t \in (0,\widetilde  T)$.

Moreover,
\begin{align*}
\bint_0^h \int_\Omega |\tilde u(x,\tilde t) - \tilde g_o(x)|^{q+1}\, \d x \d \tilde t &= M^{-(q+1)} \bint_0^{M^{q-1}h} \int_\Omega | u(x, t) - g_o(x)|^{q+1}\, \d x \d t \\
&\xrightarrow{h \to 0} 0,
\end{align*}
completing the proof.
\end{proof}

\section{On the local and global notions of solution}
\label{sec:appendix_definitions}
Here we show that when the obstacle function satisfies~\eqref{eq:psi_conds}, our notion of solution includes the notion of so called global weak solutions to the obstacle problem as defined in~\cite{BLS} (see also~\cite{Schaetzler1,Schaetzler2}). We point out that the results also hold true under the assumptions $\psi \in L^2(0,T;H^1(\Omega)) \cap C([0,T]; L^{q+1}(\Omega))$ and $\partial_t \psi \in L^{q+1}(\Omega_T)$. On the other hand, by \cite[Lemma 3.5]{BLS} any weak solution attaining the initial and boundary datum $g$ in the sense of Definition~\ref{d.obstacle-wsol} (1)-(3) is a global solution, provided that $\R^n \setminus \Omega$ is uniformly 2-thick and $g \in K'_{\psi}(\Omega_T) \cap C([0,T];L^{q+1}(\Omega))$ with $g(\cdot,0) = g_o \in I_{\psi_o}(\Omega)$.

First, defining the class of functions
$$
K'_{\psi,g}(\Omega_T) := \left\{ v \in K_{\psi,g}(\Omega_T): \partial_t v \in L^{q+1}(\Omega_T) \right\},
$$
a global weak solution to the obstacle problem~\eqref{eq:pde} is given as follows.

\begin{definition} \label{def:obst-gsol}
Suppose that $g\in K'_{\psi}(\Omega_T) \cap C([0,T];L^{q+1}(\Omega))$ and $g(\cdot,0 ) = g_o \in I_{\psi_o}(\Omega)$. A function $u \in K_{\psi,g}(\Omega_T) \cap C([0,T];L^{q+1}(\Omega))$ is a global weak solution to the obstacle problem if and only if
$$
\llangle \partial_t \u^q, \alpha \left( v-u \right) \rrangle_{g_o} + \iint_{\Omega_T} \alpha \mathbf{A}(x,t,\u^q, \nabla u) \cdot \nabla(v-u) \, \d x \d t \geq 0
$$
for every cut-off function $\alpha\in W^{1,\infty}([0,T]; \R_{\geq 0})$ with $\alpha(T) = 0$ and all comparison maps $v \in K'_{\psi,g}(\Omega_T) \cap C([0,T];L^{q+1}(\Omega))$. Here we define the time-term by
\begin{align*}
\llangle \partial_t \u^q , \alpha(v-u) \rrangle_{g_o} := &\iint_{\Omega_T} \alpha' \left( \tfrac{q}{q+1}|u|^{q+1} -  \u^q v \right) - \alpha  \u^q \partial_t v \, \d x \d t \\
&\phantom{=} + \alpha(0) \int_\Omega \tfrac{q}{q+1} |g_o|^{q+1} - \boldsymbol{g}_o^q v(\cdot,0) \, \d x.
\end{align*}
\end{definition}

First we show that initial values are attained in the sense of Definition~\ref{d.obstacle-wsol} (3). To this end, let us define
$$
	\mathfrak{g}_q(v,u)
	:=
	q \int_u^v |s|^{q-1} (s-u) \,\ds
	=
	\tfrac{q}{q+1} \big( |v|^{q+1} - |u|^{q+1} \big) - u \big( \boldsymbol{v}^q - \u^q \big),
$$
such that 
$$
\mathfrak{g}_\frac{1}{q}(\boldsymbol{v}^q,\u^q) := \tfrac{1}{q+1} |v|^{q+1} - \tfrac{1}{q+1} |u|^{q+1} - \u^q (v-u).
$$
Observe that if $q > 1$, for any $a,b,c \in \R$ by~\cite[Lemma 2.2 and 2.3]{Schaetzler2} we have that
\begin{equation} \label{eq:I-estimate}
|a-b|^{q+1} \leq \gamma(q) \left( \mathfrak{g}_\frac{1}{q}(\boldsymbol{c}^q, \boldsymbol{a}^q) + |c-b|^{q+1}\right).
\end{equation}
If $0<q<1$, by using also~\cite[Lemma 2.1]{BDL} we find that
\begin{equation} \label{eq:I-estimate-deg}
|a-b|^{q+1} \leq \gamma(q) \left( (|a|+|c|)^{(1-q)(1+q)}\mathfrak{g}_\frac{1}{q}^q(\boldsymbol{c}^q, \boldsymbol{a}^q) + |c-b|^{q+1} \right).
\end{equation}

In the next lemma, we show that a global weak solution in the sense of Definition~\ref{def:obst-gsol} attains initial values in the sense of~\eqref{eq:obstacle-initial-q+1} in Definition~\ref{d.obstacle-wsol}.

\begin{lemma} \label{lem:initial-values-impl}
A global weak solution to the obstacle problem according to Definition~\ref{def:obst-gsol} attains initial values in the sense of~\eqref{eq:obstacle-initial-q+1} in Definition~\ref{d.obstacle-wsol}.
\end{lemma}

\begin{proof}
Fix $\eps > 0$ and let $\alpha$ be a piecewise affine function such that $\alpha(t) = (\eps - t)/\eps$ in $[0,\eps)$ and $\alpha \equiv 0$ in $[\eps, T)$. We use $g$ as comparison map. In particular, we have
\begin{align*}
- \iint_{\Omega_T} \alpha \u^q \partial_t g \, \d x \d t 
&= -  \iint_{\Omega_T} \alpha (\u^q- \g^q) \partial_t g \, \d x \d t -  \iint_{\Omega_T} \alpha \g^q \partial_t g \, \d x \d t \\
&= -  \iint_{\Omega_T} \alpha (\u^q- \g^q) \partial_t g \, \d x \d t + \tfrac{1}{q+1} \iint_{\Omega_T} \alpha' |g|^{q+1} \, \d x \d t \\
&\phantom{+} + \tfrac{1}{q+1} \alpha(0) \int_\Omega |g_o|^{q+1} \, \d x.
\end{align*}
In total we have
\begin{align*}
- \iint_{\Omega_T} \alpha' \mathfrak{g}_\frac{1}{q}(\g^q,\u^q) \, \d x \d t &\leq \iint_{\Omega_T} \alpha \mathbf{A}(x,t,\u^q, \nabla u) \cdot \nabla(g-u) \, \d x \d t \\
&\phantom{+} + \iint_{\Omega_T} \alpha (\g^q - \u^q) \partial_t g \, \d x \d t.
\end{align*}
Lebesgue's dominated convergence theorem implies that 
\begin{equation*}
\lim_{\eps \to 0} \bint_0^\eps \int_\Omega \mathfrak{g}_\frac{1}{q}(\g^q,\u^q) \, \d x \d t \leq 0.
\end{equation*}
In case $q > 1$~\eqref{eq:I-estimate} yields
\begin{align*}
\bint_0^\eps \int_\Omega |u-g_o|^{q+1} \, \d x \d t 
&\leq c(q) \bint_0^\eps \int_\Omega \mathfrak{g}_\frac{1}{q}(\g^q,\u^q) \, \d x \d t \\
&\phantom{+}+ c(q)\bint_0^\eps \int_\Omega |g-g_o|^{q+1} \, \d x \d t\\
&\xrightarrow{\eps \to 0} 0
\end{align*}
by using also the fact $g\in C([0,T];L^{q+1}(\Omega))$ with $g(\cdot,0 ) = g_o$.

In case $0<q<1$ we use~\eqref{eq:I-estimate-deg} together with H\"older's inequality and obtain
\begin{align*}
\bint_0^\eps \int_\Omega |u-g_o|^{q+1} \, \d x \d t 
&\leq c(q) \left(\bint_0^\eps \int_\Omega \left( |u| + |g| \right)^{q+1}  \, \d x \d t \right)^{1-q} \\
&\phantom{+}\cdot \left( \bint_0^\eps \int_\Omega \mathfrak{g}_\frac{1}{q}(\g^q,\u^q) \, \d x \d t \right)^q \\
&\phantom{+} + c(q) \bint_0^\eps \int_\Omega |g-g_o|^{q+1} \, \d x \d t\\
&\xrightarrow{\eps \to 0} 0.
\end{align*}
By using also that $u \in C([0,T];L^{q+1}(\Omega))$ the proof is completed.
\end{proof}

\begin{lemma} \label{lem:global-is-local}
Suppose that $\Omega \subset \R^n$ is a bounded open set which satisfies~\eqref{geometry}. Let $g$ be as in Definition~\ref{def:obst-gsol} and $\psi \in C(\overline{\Omega_T})$. If $u$ is a global weak solution according to Definition~\ref{def:obst-gsol}, then $u$ is a weak solution according to Definition~\ref{d.obstacle-wsol} satisfying properties (1) and (2).
\end{lemma}

\begin{proof}
Let us denote $\Omega_{\eps} = \{x \in \Omega : \dist(x,\partial \Omega) > \eps\}$. Fix $\eta \in C^1_0(\Omega; [0,1])$ and choose $\eps>0$ so small that $\spt(\eta) \subset \Omega_{\eps}$. Let $\varphi_\eps \in C^\infty_0(\Omega;[0,1])$ such that $\varphi_\eps \equiv 1$ in $\Omega_\eps$ and $|\nabla \varphi_{\eps}| \leq \tfrac{c}{\eps}$ for a numerical constant $c > 0$.
We consider a sequence $(\delta_i)_{i \in \N}$ with $0 < \delta_i \downarrow 0$ as $i \to \infty$ and set $h_i := \delta_i^{2(n+2)}$.
Then, let $R > 0$ be such that $B_\frac{R}{2}(0) \supset \Omega$. We extend $\psi(\cdot,0)$ as a continuous function to $B_R(0) \setminus \overline{\Omega}$ and define $\psi(\cdot,0) \equiv 0$ in $\R^n \setminus B_R(0)$.  
Let us define standard mollifications of these extensions with parameter $\delta_i$ by
$$
	\psi_{o,\delta_i} := \psi(\cdot,0) \ast \phi_{\delta_i},
$$
where $\phi_\delta(x) := \delta^{-n} \phi\left(\frac{x}{\delta}\right)$ denotes a standard mollifier in $\R^n$.
Further, we define mollifications in time
$$
	u_i := \mollifytime{u}{h_i}
	\quad\text{and}\quad
	\psi_i := \mollifytime{\psi}{h_i}
$$
according to \eqref{e.time_moll_g} with initial values $\psi_{o,\delta_i}$.
By construction, we have that $u_i \geq \psi_i$ for all $i \in \N$.
Moreover, by Lemma \ref{lem:time_mollification} we find that $u_i \in C([0,T];L^{q+1}(\Omega)) \cap L^2(0,T;H^1(\Omega))$ with time derivative $\partial_t u_i \in L^{q+1}(\Omega_T)$.
Next, fix any comparison map $v \in K'_{\psi}(\Omega_T)$ and choose
$$
v_i = \eta v + (\varphi_\eps - \eta) (u_i + \| \psi- \psi_i\|_\infty) +  (1 - \varphi_\eps) g.
$$
Observe that $v_i \in K'_{\psi,g}(\Omega_T)$. In particular, we have that $v_i \geq \psi$, since $v \geq \psi$, $u_i + \| \psi - \psi_i\|_\infty \geq \psi$ and $g \geq \psi$. Furthermore, there clearly holds $(v_i - g)(\cdot,t) \in H^1_0(\Omega)$ for a.e. $t \in (0,T)$ by definition.

In order to treat the divergence part in the variational inequality, we claim that
\begin{equation}
	\nabla u_i \to \nabla u \quad\text{in $L^2(0,T;H^1(\Omega))$ as $i \to \infty$.}
	\label{eq:aux_convergence1}
\end{equation}
To this end, consider the mollification according to \eqref{e.time_moll_g} with zero initial values, i.e.~$\tilde{u}_i := u_i - e^{-\frac{t}{h_i}} \psi_{o,\delta_i}$.
By Lemma~\ref{lem:time_mollification}, we find that
\begin{equation}
	\nabla\tilde{u}_i \to \nabla u \quad\text{in $L^2(0,T;H^1(\Omega))$ as $i \to \infty$.}
	\label{eq:aux_convergence2}
\end{equation}
Further, observe that
$$
	\iint_{\Omega_T} \Big|\nabla \Big(e^{-\frac{t}{h_i}} \psi_{o,\delta_i} \Big)\Big|^2 \,\dx\dt
	=
	\int_0^T e^{-\frac{2t}{h_i}} \,\dt \int_\Omega |\nabla \psi_{o,\delta_i}|^2 \,\dx
	\leq
	\tfrac{h_i}{2} \int_\Omega |\nabla \psi_{o,\delta_i}|^2 \,\dx.
$$
Estimating the integral on the right-hand side of the preceding inequality by Young's inequality for convolutions, we conclude that
\begin{align*}
	\int_\Omega |\nabla \psi_{o,\delta_i}|^2 \,\dx
	\leq
	\|\psi_o\|_{L^1(B_R(0))}^2 \|\nabla \phi_{\delta_i}\|_{L^2(\R^n)}^2
	\leq
	\delta_i^{-(n+2)} \|\psi_o\|_{L^1(B_R(0))}^2 \|\nabla \phi\|_{L^2(\R^n)}^2.
\end{align*}
Combining the preceding two inequalities and recalling the definition of $\delta_i$, we obtain that
\begin{align*}
	\iint_{\Omega_T} \Big|\nabla \Big(e^{-\frac{t}{h_i}} \psi_{o,\delta_i} \Big)\Big|^2 \,\dx\dt
	&\leq
	\tfrac{1}{2} \|\psi_o\|_{L^1(B_R(0))}^2 \|\nabla \phi\|_{L^2(\R^n)}^2 h_i \delta_i^{-(n+2)} \\
	&=
	\tfrac{1}{2} \|\psi_o\|_{L^1(B_R(0))}^2 \|\nabla \phi\|_{L^2(\R^n)}^2 \sqrt{h_i}
	\to 0
\end{align*}
in the limit $i \to \infty$.
Together with \eqref{eq:aux_convergence2} this implies \eqref{eq:aux_convergence1}.

By observing that $1 = \eta + (\varphi_\eps - \eta) + (1- \varphi_\eps)$ everywhere in $\Omega$, for the divergence part we have
\begin{align*}
\iint_{\Omega_T} &\alpha \mathbf{A}(x,t,\u^q,\nabla u) \cdot \nabla (v_i-u) \, \d x\d t \\
&= 
\iint_{\Omega_T} \alpha \mathbf{A}(x,t,\u^q,\nabla u) \cdot \nabla (\eta (v-u) ) \, \d x\d t \\
&\phantom{+}+ \iint_{\Omega_T} \alpha \mathbf{A}(x,t,\u^q,\nabla u) \cdot \nabla [ (\varphi_\eps -\eta) (u_i-u + \| \psi - \psi_i \|_\infty) ] \, \d x\d t \\
&\phantom{+} + \iint_{\Omega_T} \alpha \mathbf{A}(x,t,\u^q,\nabla u) \cdot \nabla ((1-\varphi_\eps) (g-u) ) \, \d x\d t  \\
&\xrightarrow{i \to \infty} \iint_{\Omega_T} \alpha \mathbf{A}(x,t,\u^q,\nabla u) \cdot \nabla (\eta (v-u) ) \, \d x\d t \\
&\phantom{+} + \iint_{\Omega_T} \alpha \mathbf{A}(x,t,\u^q,\nabla u) \cdot \nabla ((1-\varphi_\eps) (g-u) ) \, \d x\d t,
\end{align*}
by \eqref{eq:aux_convergence1} and since $u \in L^2(0,T;H^1(\Omega))$, $\psi \in C(\overline{\Omega_T})$ and $|\mathbf{A}(x,t,\u^q,\nabla u)| \in L^2(\Omega_T)$ together with Lemma~\ref{lem:time_mollification} (1) and (4). For the last term on the right hand side we obtain that
\begin{align*}
\iint_{\Omega_T} &\alpha \mathbf{A}(x,t,\u^q,\nabla u) \cdot \nabla ((1-\varphi_\eps) (g-u) ) \, \d x\d t \\
&= \iint_{\Omega_T} \alpha (1-\varphi_\eps) \mathbf{A}(x,t,\u^q,\nabla u) \cdot \nabla (g-u) \, \d x\d t \\
&\phantom{+} + \iint_{\Omega_T} \alpha (u-g) \mathbf{A}(x,t,\u^q,\nabla u) \cdot \nabla \varphi_\eps   \, \d x\d t.
\end{align*}
The first term on the right hand side will vanish when $\eps \downarrow 0$ by Lebesgue's dominated convergence theorem since $\varphi_\eps \uparrow \chi_{\Omega}$ pointwise. For the second term we have that
\begin{align*}
 \iint_{\Omega_T} &\alpha (u-g) \mathbf{A}(x,t,\u^q,\nabla u) \cdot \nabla \varphi_\eps   \, \d x\d t \\
&\leq
C_1\iint_{\Omega_T} \alpha |u-g| |\nabla u| |\nabla \varphi_\eps|  \, \d x\d t \\
&\leq 
c(C_1) \|\alpha\|_\infty \|\nabla u \chi_{\Omega \setminus \Omega_\eps}\|_{L^2} \left( \iint_{(\Omega \setminus \Omega_\eps) \times (0,T)} \frac{|u-g|^2}{\eps^2}  \, \d x\d t  \right)^\frac12 \\
&\leq 
c(C_1)\|\alpha\|_\infty \|\nabla u \chi_{\Omega \setminus \Omega_\eps}\|_{L^2} \left( \iint_{(\Omega \setminus \Omega_\eps) \times (0,T)} \frac{|u-g|^2}{\dist(x,\partial \Omega)^2}  \, \d x\d t  \right)^\frac12 \\
&\leq 
c(C_1)\|\alpha\|_\infty \|\nabla u \chi_{\Omega \setminus \Omega_\eps}\|_{L^2} \left( \iint_{\Omega\times (0,T)} \frac{|u-g|^2}{\dist(x,\partial \Omega)^2}  \, \d x\d t  \right)^\frac12 \\
&\leq 
c(C_1,n,\alpha_*)\|\alpha\|_\infty \|\nabla u \chi_{\Omega \setminus \Omega_\eps}\|_{L^2} \left( \iint_{\Omega\times (0,T)} |\nabla(u-g)|^2  \, \d x\d t  \right)^\frac12
\xrightarrow{\eps \downarrow 0} 0,
\end{align*}
by using Hardy's inequality from Lemma~\ref{lem:hardy} in the last line. The convergence holds by the dominated convergence theorem, since $|\nabla u| \in L^2(\Omega_T)$ and $\chi_{\Omega \setminus \Omega_\eps} \xrightarrow{\eps \downarrow 0} 0$ pointwise everywhere in $\Omega$.

Furthermore, using that
$$
	\alpha (\varphi_\epsilon - \eta) \big( \u^q - \u_i^q \big) \partial_t u_i
	=
	\tfrac{1}{h_i}\alpha (\varphi_\epsilon - \eta) \big( \u^q - \u_i^q \big) (u-u_i)
	\geq 0
$$
by construction of the mollification and that $\u_i^q \partial_t u_i = \partial_t |u_i|^{q+1}$, integrating by parts yields the estimate
\begin{align*}
- &\iint_{\Omega_T} \alpha \u^q \partial_t v_i \, \d x \d t \\
&= 
- \iint_{\Omega_T} \alpha \u^q\left( \eta \partial_t v + (\varphi_\eps - \eta) \partial_t u_i +  (1 - \varphi_\eps) \partial_t g  \right) \, \d x \d t  \\
&\leq - \iint_{\Omega_T} \alpha \eta\u^q  \partial_t v  + \tfrac{1}{q+1} \alpha (\varphi_\eps - \eta) \partial_t |u_i|^{q+1}\, \d x \d t \\
&\phantom{+} - \iint_{\Omega_T} \alpha (1 - \varphi_\eps) \partial_t g  \, \d x \d t  \\
&\xrightarrow{i \to \infty} \iint_{\Omega_T}- \alpha \eta\u^q  \partial_t v  + \tfrac{1}{q+1} \alpha' (\varphi_\eps - \eta)  |u|^{q+1}\,  \d x \d t  - \iint_{\Omega} \alpha (1 - \varphi_\eps) \partial_t g  \, \d x \d t. 
\end{align*}
Furthermore, we have 
\begin{align*}
\iint_{\Omega_T} \alpha' \u^q v_i \, \d x\d t \xrightarrow{i \to \infty} \iint_{\Omega_T} \alpha' \eta \u^q v + \alpha' (\varphi_\eps - \eta)|u|^{q+1}  + \alpha' (1-\varphi_\eps) \u^q g  \, \d x \d t.
\end{align*}
Finally, by inserting the preceding estimates into the time term and passing to the limit $\eps \downarrow 0$ by means of the dominated convergence theorem, we conclude the claim of the lemma.
\end{proof}

By combining Lemmas~\ref{lem:initial-values-impl} and~\ref{lem:global-is-local} we arrive at the following result.

\begin{lemma} \label{lem:global-is-local-final}
Let $\Omega \subset \R^n$ be a bounded open set which satisfies~\eqref{geometry}. Suppose that $\psi \in C(\overline{\Omega_T})$ and that $g$ and $g_o$ satisfy the assumptions in Definition~\ref{def:obst-gsol}. Let $u$ be a global weak solution according to the aforementioned definition. Then $u$ is a weak solution to the obstacle problem according to Definition~\ref{d.obstacle-wsol} satisfying (1), (2) and (3).
\end{lemma}

\begin{remark}
Observe that the assumption~\eqref{geometry} is not the weakest possible for $\Omega$ in Lemmas~\ref{lem:global-is-local} and~\ref{lem:global-is-local-final}. Both hold true under the assumption that $\R^n\setminus \Omega$ is uniformly $2$-thick, see also Remark~\ref{rem:hardy}.
\end{remark}

\end{document}